\documentclass[11pt]{amsart}
\usepackage{amssymb,amsmath}
\usepackage[alphabetic,initials,nobysame]{amsrefs}
\usepackage{eucal}

\theoremstyle{plain}
\newtheorem{theo}{Theorem}[section]
\newtheorem{lemma}[theo]{Lemma}
\newtheorem{prop}[theo]{Proposition}
\newtheorem{cor}[theo]{Corollary}
\newtheorem*{clm}{Claim}
\theoremstyle{remark}
\newtheorem{rmrk}[theo]{Remark}
\theoremstyle{definition}
\newtheorem{dfn}[theo]{Definition}
\numberwithin{equation}{section}
\numberwithin{figure}{section}

\newcommand{\D}{\partial}
\newcommand{\R}{\mathbb{R}}

\newcommand{\la}{\lambda}
\newcommand{\vf}{\varphi}

\newcommand{\ve}{\varepsilon}
\newcommand{\Om}{\Omega}
\newcommand{\Rn}{\R^n}

\renewcommand{\div}{\operatorname{div}}

\newcommand{\loc}{\mathrm{loc}}
\newcommand{\defeq}{\stackrel{\text{def}}{=}}

\title[An epiperimetric inequality approach, etc.]
{An epiperimetric inequality approach to the regularity of the free boundary in the Signorini problem with variable coefficients}

\author{Nicola Garofalo}
\address{Dipartimento d'Ingegneria Civile e Ambientale (DICEA)\\ Universit\`a di Padova\\ via Trieste 63, 35131 Padova, Italy}
\email[Nicola Garofalo]{nicola.garofalo@unipd.it}
\thanks{First author supported in part by a grant ``Progetti d'Ateneo, 2013,'' University of Padova}

\author{Arshak Petrosyan}
\address{Department of Mathematics, Purdue University} \email[Arshak Petrosyan]{arshak@math.purdue.edu}
\thanks{Second author supported in part by the  NSF Grant DMS-1101139}

\author{Mariana Smit Vega Garcia}
\address{Universit\"at Duisburg-Essen, 
Fakult\"at f\"ur Mathematik,
45117 Essen} \email[Mariana Smit Vega Garcia]{mariana.vega-smit@uni-due.de}

\thanks{This project was completed while the third author was visiting
  the Institute Mittag-Leffler for the program ``Homogenization and
  Random Phenomenon.'' She thanks this institution for the gracious hospitality and the excellent work environment.} 
\begin{document}
\begin{abstract}
In this paper we establish the $C^{1,\beta}$ regularity of the regular
part of the free boundary in the Signorini problem for elliptic
operators with variable Lipschitz coefficients. This work is a
continuation of the recent paper \cite{GS}, where two of us
established the interior optimal regularity of the solution. Two of
the central results of the present work are a new monotonicity formula
and a new epiperimetric inequality. 
\end{abstract}
\maketitle

\tableofcontents

\section{Introduction}\label{S:intro}

\subsection{Statement of the problem and main assumptions}
The purpose of the present paper is to establish the $C^{1,\beta}$
regularity of the free boundary near so-called regular points in the
Signorini problem for elliptic operators with variable Lipschitz
coefficients. Although this work represents a continuation of the
recent paper \cite{GS}, where two of us established the interior
optimal regularity of the solution, proving the regularity of the free
boundary has posed some major new challenges. Two of the central
results of the present work are a new monotonicity formula (Theorem
\ref{T:weiss}) and a new epiperimetric inequality (Theorem
\ref{T:epi}). Both of these results have been inspired by those
originally obtained by Weiss in \cite{W} for the classical obstacle
problem, but the adaptation to the Signorini problem has required a
substantial amount of new ideas.

\emph{The lower-dimensional} (or \emph{thin}) \emph{obstacle problem}
consists of minimizing the (generalized) Dirichlet energy 
\begin{equation}\label{mp0}
\underset{u\in \mathcal K}{\min} \int_{\Omega}\langle A(x)\nabla u,\nabla u\rangle dx,
\end{equation}
where $u$ ranges in the closed convex set 
\[
\mathcal K=\mathcal K_{g,\varphi}=\{u\in W^{1,2}(\Omega)\mid u= g
\text{ on } \D\Omega,\ u\geq \varphi \text{ on }
\mathcal{M}\cap\Omega\}.
\]
Here, $\Om\subset \Rn$ is a given bounded open set, $\mathcal M$ is a
codimension one manifold which separates $\Om$ into two parts, $g$ is
a boundary datum and the function $\varphi:\mathcal M \to \R$
represents the lower-dimensional, or thin, obstacle. The functions $g$
and $\varphi$ are required to satisfy the standard compatibility
condition $g\ge \varphi$ on $\D \Om \cap \mathcal M$. This problem is
known also as (scalar) \emph{Signorini problem}, as the minimizers
satisfy Signorini conditions on $\mathcal{M}$ (see
\eqref{bb2i}--\eqref{bb4i} below in the case of flat $\mathcal{M}$).

Our assumptions on the matrix-valued function $x\mapsto A(x) = [a_{ij}(x)]$ in
\eqref{mp0} are that $A(x)$ is symmetric, uniformly elliptic, and
Lipschitz continuous (in short $A\in C^{0,1}$). Namely:
\begin{equation}
  \label{Asym}
  a_{ij}(x) = a_{ji}(x)\quad \text{for $i, j = 1,\ldots,n$, and every
    $x\in \Om$};
\end{equation}
 there exists $\la>0$ such that for every $x\in \Om$ and $\xi\in \Rn$, one has
\begin{equation}\label{A0}
\lambda|\xi|^2\leq \langle A(x)\xi,\xi\rangle \leq\lambda^{-1}|\xi|^2;
\end{equation}
there exists $Q\ge 0$ such that
\begin{equation}\label{A}
|a_{ij}(x) - a_{ij}(y)| \le Q |x-y|,\quad x, y\in \Om.
\end{equation}
By standard methods in the calculus of variations it is known that,
under appropriate assumptions on the data, the minimization problem
\eqref{mp0} admits a unique solution $u\in \mathcal K$, see e.g.\
\cite{F2}, or also \cite{T}. The set
\[
\Lambda^\vf(u)=\{x\in\mathcal{M}\cap \Om \mid u(x)=\varphi(x)\}
\]
is known as the \emph{coincidence set}, and its boundary (in the relative topology of $\mathcal M$)
\[
\Gamma^\vf(u)=\D_{\mathcal{M}}\Lambda^\vf(u)
\]
is known as the \emph{free boundary}. In this paper we are interested
in the local regularity properties of $\Gamma^\varphi(u)$. When
$\vf=0$, we will write $\Lambda(u)$ and $\Gamma(u)$, instead of
$\Lambda^0(u)$ and $\Gamma^0(u)$. 

We also note that since we work with Lipschitz coefficients, it is not
restrictive to consider the situation in which the thin manifold is
flat, which we take to be $\mathcal M = \{x_n = 0\}$. We thus consider
the Signorini problem \eqref{mp0} when the thin obstacle $\vf$ is
defined on   
 \[
 B_1' = \mathcal M \cap B_1= \{(x',0)\in B_1\mid |x'|<1\},
 \]
which we call the thin ball in $B_1$. In this case we will also impose
the following conditions on the coefficients
\begin{equation}\label{mjai}
a_{in}(x',0)=0\quad\text{in}\ B_1',\ \text{for}\ i<n,
\end{equation}
which essentially means that the conormal directions $A(x',0)\nu_\pm$
are the same as normal directions $\nu_\pm=\mp e_n$. We stress here
that the condition \eqref{mjai} is not restrictive as it can be
satisfied by means of a $C^{1,1}$ transformation of variables, as proved in Appendix B of \cite{GS}. 

We assume that $\varphi\in C^{1,1}(B_1')$ and denote by $u$ the unique
solution to the minimization problem \eqref{mp0}. (Notice that, by
letting $\tilde \vf(x',x_n) = \vf(x')$, we can think of $\vf\in
C^{1,1}(B_1)$, although we will not make such distinction explicitly.)
Such $u$ satisfies
\begin{align}\label{bb1i}
Lu = \div(A\nabla u)=0&\quad\text{in}\  B_1^+\cup B_1^-,\\ 
\label{bb2i}
u-\vf \geq 0&\quad\text{in}\  B_1',\\
\label{bb3i}
\langle A\nabla u,\nu_+\rangle +\langle A\nabla u,\nu_-\rangle \geq
0&\quad \text{in}\  B_1',\\
\label{bb4i}
(u-\vf) (\langle A\nabla u,\nu_+\rangle +\langle A\nabla
u,\nu_-\rangle )=0&\quad\text{in}\  B_1'.
\end{align}
The conditions \eqref{bb2i}--\eqref{bb4i} are known as \emph{Signorini
  \emph{or} complementarity conditions}. It has been recently shown in
\cite{GS} that, under the assumptions above, the unique solution of 
\eqref{bb1i}--\eqref{bb4i} is in $C^{1,1/2}(B_1^\pm\cup
B_1')$. This regularity is optimal since the function $u(x) = \Re
(x_1+i|x_n|)^{3/2}$ solves the Signorini problem for the Laplacian and
with thin obstacle $\varphi \equiv 0$.

Henceforth, we assume that $0$ is a free boundary point,
i.e., $0\in\Gamma^\vf(u)$, and we suppose without restriction that
$A(0) = I$. Under these hypothesis, we consider the 
following \emph{normalization} of $u$
\begin{equation}\label{vi}
v(x)=u(x)-\varphi(x') + b x_n,\quad b=\partial_{\nu_+} u(0).
\end{equation}
Note that $\partial_{\nu_+}u(0)+\partial_{\nu_-}u(0)=0$,
which follows from \eqref{bb4i}, by taking a limit  from inside the set
$\{(x',0)\in B_1'\mid u(x', 0)>\varphi(x')\}$, and from the hypothesis
$A(0) = I$. This implies that 
\begin{equation}\label{c00i}
v(0)=|\nabla v(0)|=0.
\end{equation}
Next, note that, in view of \eqref{A} above and of the assumption $\vf\in
C^{1,1}(B_1)$, we have
\begin{equation*}\label{c.5.5i}
-L(\varphi(x')- b x_n)\defeq f\in L^\infty(B_1).
\end{equation*}
Hence, we can rewrite \eqref{bb1i}--\eqref{bb4i} in terms of $v$ as follows:
\begin{align}\label{c11i}
Lv = \div(A\nabla v)=f&\quad\text{in}\  B_1^+\cup B_1^-,\quad
f\in L^{\infty}(B_1),\\
\label{c22i}
v \geq 0&\quad\text{in}\  B_1',\\
\label{c33i}
\langle A\nabla v,\nu_+\rangle +\langle A\nabla v,\nu_-\rangle \geq
  0&\quad\text{in}\  B_1',\\
\label{c44i}
v(\langle A\nabla v,\nu_+\rangle +\langle A\nabla v,\nu_-\rangle
)=0&\quad\text{in}\ B_1'.
\end{align}
Note that from \eqref{c11i} we have
\begin{equation}\label{c55i}
\int_{B_1}\left(\langle A\nabla v,\nabla\eta\rangle + f
  \eta\right)=\int_{B_1'}(\langle A\nabla v,\nu_+\rangle +\langle
A\nabla v,\nu_-\rangle )\eta,\quad\eta\in C^{\infty}_0(B_1).
\end{equation}
and thus the conditions \eqref{c22i}--\eqref{c44i} imply that $v$
satisfies the variational inequality 
$$
\int_{B_1}\left(\langle A\nabla v,\nabla (w-v)\rangle + f (w-v)\right)\geq
0,\quad\text{for any }w\in\mathcal{K}_{v,0},
$$ 
where $\mathcal{K}_{v,0}=\{w\in W^{1,2}(B_1) \mid  w= v \text{ on } \D
B_1, w\geq 0 \text{ on } B_1'\}$. Since $\mathcal{K}_{v,0}$ is convex,
and so is the energy functional
\begin{equation}\label{energy-with-f}
\int_{B_1}\left(\langle A\nabla w,\nabla w\rangle + 2fw\right),
\end{equation}
this is equivalent to saying that
$v$ minimizes \eqref{energy-with-f} among all functions $w\in
\mathcal{K}_{v,0}$. Notice as well that 
\begin{align*}
\Lambda^\vf(u)&=\{u(\cdot, 0)=\varphi\}=\{v(\cdot,0)=0\}=\Lambda^0(v)
= \Lambda(v).
\end{align*}
We will write $\Gamma^\vf(u)=\Gamma^0(v)=\Gamma(v)$, and thus $0\in
\Gamma(v)$ now and \eqref{c00i} holds.

\subsection{Main result}

To state the main result of this paper, we need to further classify
the free boundary points. This is achieved by means of the
\emph{truncated frequency function}
\[
N(r)=N_L(v,r)\defeq
\frac{\sigma(r)}{2}e^{K'r^{\frac{1-\delta}{2}}}
\frac{d}{dr}\log\max\left\{\frac{1}{\sigma(r)r^{n-2}}\int_{S_r}
  v^2\mu, r^{3+\delta}\right\}.
\]
Here, $\mu(x)=\langle A(x)x,x \rangle/|x|^2$ is a conformal factor,
$\sigma(r)$ is an auxiliary function with the property that 
$\sigma(r)/r\to \alpha>0$ as $r\to 0$, $0<\delta<1$, and $K'$ is
a universal constant (see Section~\ref{S:not} for
exact definitions and properties). This function was introduced in
\cite{GS}, and represents a version of Almgren's
celebrated frequency function (see \cite{A}), adjusted for the solutions of
\eqref{c11i}--\eqref{c44i}. By  Theorem~\ref{T:lym} in \cite{GS},
$N(r)$ is monotone increasing and hence 
the limit
$$
\tilde{N}(0+)=\lim_{r\to 0}\tilde{N}(r),\quad\text{where}\quad
\tilde{N}(r)=\frac{r}{\sigma(r)}N(r)
$$
exists.  The remarkable fact is that either $\tilde{N}(0+)=3/2$, or
$\tilde{N}(0+)\ge (3+\delta)/2$, see Lemma~\ref{L:newcrucial}
below. This leads to the following definition.

\begin{dfn} We say that $0\in\Gamma(v)$ is a \emph{regular} point iff
  $\tilde{N}(0+)=3/2$. Shifting the origin to $x_0\in \Gamma(v)$, and denoting
  the corresponding frequency function
by  $\tilde{N}_{x_0}$, we define
$$
\Gamma_{3/2}(v)=\{x_0\in \Gamma(v)\mid \tilde N_{x_0}(0+)=3/2\},
$$
the set of all regular free boundary points, also known as the
\emph{regular set}.
\end{dfn}

The remaining part of the free boundary is divided into the sets
$\Gamma_{\kappa}(v)$, according to the corresponding value of
$\tilde{N}(0+)\defeq \kappa$. We note that the range of
possible values for $\kappa$ can be further refined, provided more
regularity is known for the coefficients $A(x)$. This can be achieved
by replacing the truncation function $r^{3+\delta}$ in the formula
for $N(r)$ with higher powers of $r$, similarly to what was done in
\cite{GP} in the case of the Laplacian. This will provide more
information on the set of possible values of $\tilde{N}(0+)$ which
will serve as a classification parameter.

\medskip
The following theorem is the central result of this paper.

\begin{theo}\label{main} Let $v$ be a solution of
  \eqref{c11i}--\eqref{c44i} with $x_0\in \Gamma_{3/2}(v)$. Then,
  there exists $\eta_0>0$, depending on $x_0$, such that, after a
  possible rotation of coordinate axes in $\R^{n-1}$, one has
  $B'_{\eta_0}(x_0)\cap \Gamma(v)\subset \Gamma_{3/2}(v)$, and
$$
B'_{\eta_0}\cap \Lambda(v)=B'_{\eta_0}\cap \{x_{n-1}\leq g(x_1,\ldots,
x_{n-2})\}
$$
for $g\in C^{1,\beta}(\R^{n-2})$ with a universal exponent
$\beta\in(0,1)$.
\end{theo}

This result is known and well-understood in the case $L=\Delta$,
see \cite{ACS} or Chapter 9 in \cite{PSU}. However, the existing
proofs are based on differentiating the equation for $v$ in tangential
directions $e\in \R^{n-1}$ and establishing the nonnegativity of $\D_ev$
in a cone of directions, near regular free boundary points
(directional monotonicity). This implies the Lipschitz regularity of $\Gamma_{3/2}(v)$, which can be
pushed to $C^{1,\beta}$ with the help of the boundary Harnack
principle. The idea of the directional monotonicity goes back at least
to the paper \cite{Alt}, while the application of the boundary Harnack
principle originated in \cite{AC1}; see also
\cite{C2} and the book \cite{PSU}. In the case $L=\Delta$, we
also want to mention two recent papers that prove the smoothness of the
regular set: Koch, the second author, and Shi \cite{KPS} establish the
real analyticity of $\Gamma_{3/2}$ by using hodograph-type transformation and
subelliptic estimates, and De Silva and Savin \cite{DS2} prove
$C^\infty$ regularity of  $\Gamma_{3/2}$  by higher-order boundary
Harnack principle in slit domains.

Taking directional derivatives, however, does not work well for
the problem studied in this paper, particularly so since we are
working with solutions of the non-homogeneous equation \eqref{c11i},
which corresponds to nonzero thin
obstacle $\varphi$. In contrast, the methods in this paper are purely
energy based, and they are new even in the case of the thin obstacle problem
for the Laplacian. They are inspired by the homogeneity improvement
approach of Weiss \cite{W} in the classical obstacle problem. The
latter consists of a combination of a monotonicity formula and an 
epiperimetric inequality. In this connection we mention that recently
in \cite{FGS}, Focardi, Gelli, and Spadaro 
extended Weiss' method to the classical obstacle problem for operators with Lipschitz coefficients.  
We also mention a recent preprint by Koch, R\"uland, and Shi
\cite{KRS1} in which they use Carleman estimates to establish the
almost optimal interior regularity of the solution in the variable
coefficient 
Signorini problem, when the coefficient matrix is $W^{1,p}$, with
$p>n+1$. In a personal communication \cite{KRS2} these authors have
informed us of work in progress on the optimal interior regularity as
well as the $C^{1,\beta}$ regularity of the regular set for $W^{1,p}$
coefficients with $p>2(n+1)$. Their preprint was not available to us
when the present paper was completed. 

We next describe our proof of Theorem~\ref{main} above. The first main
ingredient consists of the ``almost monotonicity'' of the Weiss-type functional
\[
W_L(v,r) = \frac{1}{r^{n+1}}\int_{B_r}[\langle A(x)\nabla v, \nabla
v\rangle +vf] - \frac{3/2}{r^{n+2}}\int_{S_r}v^2\mu,
\]
for solutions of \eqref{c11i}--\eqref{c44i}. In Theorem~\ref{T:weiss}
below we prove that $W_L(v,r)+Cr^{1/2}$ is nondecreasing for a
universal constant $C$. Here, we were inspired by \cite{W} and
\cite{W2}, where Weiss introduced related monotonicity formulas in the
classical obstacle problem. In \cite{GP} two of us also proved a
similar monotonicity formula in 
  the Signorini setting, in the case of the Laplacian. In the present
  paper we use the machinery established in \cite{GS} to treat the
  case of variable Lipschitz coefficients. 
We mention that the geometric meaning of the functional $W_L$ above is
that it measures the closeness of the 
solution $v$ to the prototypical homogeneous solutions of degree
$3/2$, i.e., the functions 
$$
a\Re(\langle x',\nu\rangle+i|x_n|)^{3/2},\quad a\geq 0,\ \nu\in S_1.
$$

The second central ingredient in the proof is the epiperimetric inequality for
the functional 
$$
W(v)=W_\Delta(v,1)=\int_{B_1}|\nabla v|^2-\frac32\int_{S_1}v^2, 
$$
which states that if a $(3/2)$-homogeneous function $w$, nonnegative
on $B_1'$, is close to the
solution $h(x)=\Re(x_1+i|x_n|)^{3/2}$ in $W^{1,2}(B_1)$-norm, then
there exists $\zeta$ in $B_1$ with $\zeta=w$ on $\partial B_1$
such that
$$
W(\zeta)\leq (1-\kappa)W(w),
$$ 
for a universal $0<\kappa<1$, see Theorem~\ref{T:epi} below. 

The combination of Theorems \ref{T:weiss} and Theorem~\ref{T:epi}
provides us with a powerful tool for esta\-blishing the following
geometric rate of decay for the Weiss functional:
$$
W_L(v,r)\leq C r^\gamma,
$$
for a universal $\gamma>0$. In turn, this ultimately implies that
$$
\int_{S_1'} |v_{\bar x,0}-v_{\bar y,0}|\leq C|\bar x-\bar y|^\beta
$$ 
for properly defined homogeneous blowups $v_{\bar x,0}$ and $v_{\bar y,0}$ at
$\bar x, \bar y\in \Gamma_{3/2}(v)$. This finally implies the
$C^{1,\beta}$ regularity of $\Gamma_{3/2}(v)$ in a more or less
standard fashion.

\subsection{Structure of the paper}
The paper is organized as follows. 

\begin{itemize}
\item In Section \ref{S:not} we
recall those definitions and results from \cite{GS} which constitute
the background results of this paper. 

\item In Section \ref{S:reg} we give a more in depth
look at regular free boundary points
and prove some preliminary but important properties such as the
relative openness of the regular set $\Gamma_{3/2}(v)$ in $\Gamma(v)$
and the local uniform convergence of the truncated frequency function
$\tilde{N}_{\bar{x}}(r)\rightarrow
3/2$ on $\Gamma_{3/2}(v)$. We also introduce Almgren type scalings for
the solutions, see \eqref{Almgrentype}, which play an important
role in Section \ref{S:final}.

\item In Section \ref{S:weiss} we establish the first main technical tool
  of this paper, the Weiss-type monotonicity formula discussed above
  (Theorem~\ref{T:weiss}). This result is instrumental to studying the
  homogeneous
  blowups of our function, which we do in Section~\ref{S:hbu}.

\item Section \ref{S:epi} is devoted to proving the second main
  technical result of this paper, the epiperimetric inequality (Theorem
  \ref{T:epi}) which we have discussed above. 

\item Finally, in Section~\ref{S:final} we combine the monotonicity
  formula and the epiperimetric inequa\-lity to prove the main result
  of this paper, the $C^{1,\beta}$-regularity of the regular set
  $\Gamma_{3/2}(v)$ (Theorem~\ref{main}).
\end{itemize}

\section{Notation and preliminaries}\label{S:not}
\subsection{Basic notation}Throughout the paper we use following
notation. We work in the
Euclidean space $\R^n$, $n\geq 2$. We write the points of
$\R^n$ as $x=(x',x_n)$, where
$x'=(x_1,\ldots,x_{n-1})\in\R^{n-1}$. Very often, we identify the
points $(x',0)$ with $x'$, thus identifying the ``thin'' space
$\R^{n-1}\times\{0\}$ with $\R^{n-1}$.

For $x\in \R^n$, $x'\in\R^{n-1}$ and $r>0$, we define the ``solid'' and
``thin'' balls 
\begin{alignat*}{2}
B_r(x)&=\{y\in\R^n \mid |x-y|<r\},&\quad  B'_r(x')&=\{y'\in\R^{n-1}
\mid |x'-y'|<r\}\\
\intertext{as well as the corresponding spheres}
S_r(x)&=\{y\in\R^n \mid |x-y|=r\},& S'_r(x')&=\{y'\in\R^{n-1} \mid |x'-y'|=r\}.
\end{alignat*}
We typically do not indicate the center, if it is the origin. Thus,
$B_r=B_r(0)$, $S_r=S_r(0)$, etc.
We also denote
$$
B_r^\pm(x',0)=B_r(x',0)\cap \{\pm x_n>0\},\quad \R^n_\pm=\R^n\cap \{\pm x_n>0\}.
$$
For a given direction $e$, we denote the corresponding directional
derivative by
$$
\partial_e u=\langle \nabla u,e\rangle,
$$
whenever it makes sense. For the standard coordinate directions
$e=e_i$, $i=1,\ldots,n$, we also abbreviate $\partial_i u=\partial_{e_i} u$.

In the situation when a domain $\Omega\subset\R^n$ is divided by a
manifold $\mathcal{M}$ into two subdomains $\Omega_+$ and $\Omega_-$, 
$\nu_+$ and $\nu_-$ stand
for the exterior unit normal for $\Omega_+$ and $\Omega_-$ on
$\mathcal{M}$. Moreover, we always understand $\partial_{\nu_+}u$
($\partial_{\nu_-}u$) on $\mathcal{M}$ as
the limit from within $\Omega_+$ ($\Omega_-$). Thus, when
$\Omega_\pm=B_r^\pm$, we have
$$
\nu_\pm=\mp e_n,\quad -\partial_{\nu_+}u(x',0)=\lim_{\substack{y\to
    (x',0)\\y\in
    B_r^+}}\partial_{n}u \defeq \partial_n^+u(x',0),\quad  \partial_{\nu_-}u(x',0)=\lim_{\substack{y\to
    (x',0)\\y\in B_r^-}}\partial_{n}u\defeq \partial_n^-u(x',0)
$$

In integrals, we often do not indicate the measure of integration if it
is the Lebesgue measure on subdomains of $\R^n$, or the Hausdorff
$\mathcal{H}^{k}$ measure on manifolds of dimension $k$.

Hereafter, when we say that a constant is universal, we mean that it
depends exclusively on $n$, on the ellipticity bound $\la$ on $A(x)$,
and on the Lipschitz bound $Q$ on the coefficients
$a_{ij}(x)$. Likewise, we will say that $O(1)$, $O(r)$, etc, are
universal if $|O(1)| \le C$, $|O(r)|\le C r$, etc, with $C\ge 0$
universal.

\subsection{Summary of known results}

For the convenience of the reader, in this section we briefly recall the definitions and results proved in \cite{GS} which will be used in this paper.

As stated in Section \ref{S:intro}, we work under the nonrestrictive
situation in which the thin manifold $\mathcal{M}$ is flat. More
specifically, we consider the Signorini problem \eqref{mp0} when the
thin obstacle $\vf$ is defined on
 \[
 B_1' = \mathcal M \cap B_1= \{(x',0)\in B_1\mid |x'|<1\}.
 \]
We assume that $\varphi\in C^{1,1}(B_1')$ and denote by $u$ the unique
solution to the minimization problem \eqref{mp0}, which then satisfies \eqref{bb1i}--\eqref{bb4i}. We assume that $0$ is a free boundary point, 
i.e., $0\in\Gamma^\vf(u)$, and that $A(0) = I$, and we consider the
normalization of $u$ as in \eqref{vi}, i.e., 
$$
v(x)=u(x)-\varphi(x') + b x_n,\quad\text{with }b=\partial_{\nu_+} u(0).
$$
As remarked on Section \ref{S:intro}, $v$ satisfies
\eqref{c11i}--\eqref{c44i} with $f\defeq -L(\varphi(x')- b
x_n)\in L^\infty(B_1)$, and has the additional property that
$v(0)=|\nabla v(0)|=0$, see \eqref{c00i} above.

We then recall the following definitions from \cite{GS}. The
\emph{Dirichlet integral} of $v$ in $B_r$ is defined by
\[
D(r)=D_L(v,r)=\int_{B_r}\langle A(x)\nabla v,\nabla v\rangle ,
\] 
and the \emph{height function} of $v$ in $S_r$ is given by
\begin{equation}\label{Hv}
H(r)= H_L(v,r) = \int_{S_r}v^2 \mu , 
\end{equation}
where $\mu$ is the \emph{conformal factor}
\[
\mu(x)=\mu_L(x)=\frac{\langle A(x)x,x\rangle }{|x|^2}.
\]
We notice that, when $A(x) \equiv I$, then $\mu \equiv 1$. We also
define the \emph{generalized energy} of $v$ in $B_r$
\begin{equation}\label{Iv}
I(r) = I_L(v,r) = \int_{S_r} v \langle A\nabla v,\nu\rangle=
D_L(v,r)+\int_{B_r} vf 
\end{equation}
where $\nu$ indicates the outer unit normal to $S_r$. The following
result is Lemma 4.4 in \cite{GS}.

\begin{lemma}\label{L:H'i}
The function $H(r)$ is absolutely continuous, and for a.e.\ $r\in (0,1)$ one has
\begin{equation}\label{H'2i}
H'(r)= 2 I(r) + \int_{S_r} v^2 L|x|. 
\end{equation}
\end{lemma}

As it was explained in \cite{GS}, the second term in the right-hand
side of \eqref{H'2i} above
represents a serious difficulty to overcome if one wants to establish
the monotonicity of the generalized frequency. To bypass this
obstacle, one of the main ideas in \cite{GS} was the introduction of
the following auxiliary functions, defined for $v$ satisfying \eqref{c11i}--\eqref{c44i} and $0<r<1$:
\begin{equation}\label{psi-sig-def}
\psi(r)=e^{\int_0^r G(s)ds},\quad
\sigma(r)=\frac{\psi(r)}{r^{n-2}},\quad\text{where}\quad
G(r) = \begin{cases}
\frac{\int_{S_r} v^2 L|x|}{\int_{S_r} v^2 \mu},&\text{if}\ H(r) \not= 0,
\\
\frac{n-1}{r},&\text{if}\ H(r) = 0.
\end{cases}
\end{equation}
When $L = \Delta$, it is easy to see that $\psi(r) = r^{n-1}$ and that
$\sigma(r) = r$. We have the following simple and useful lemma which
summarizes the most relevant properties of $\psi(r)$ and $\sigma(r)$.

\begin{lemma}\label{L:psi1}
There exists a universal constant $\beta\ge 0$ such that 
\begin{equation}\label{per4}
\frac{n-1}{r}-\beta \le\frac{d}{dr} \log \psi(r)\le \frac{n-1}{r} +
\beta,\quad 0<r<1,
\end{equation}
and  one has
\begin{equation}\label{per5}
e^{-\beta(1-r)} r^{n-1}
\le  \psi(r) \le  e^{\beta(1-r)}r^{n-1},\quad 0<r<1.
\end{equation}
This implies, in particular, $\psi(0+) = 0$. In terms of the function
$\sigma(r)=\psi(r)/r^{n-2}$ we have
\begin{equation}
  \label{sig3.5}
\left| \frac{d}{dr} \log \frac{\sigma(r)}{r} \right|\leq \beta,\quad 0<r<1  
\end{equation}
and 
\begin{equation}\label{sig4}
e^{-\beta(1-r)}\le \frac{\sigma(r)}{r} \le  e^{\beta(1-r)},\quad 0<r<1.
\end{equation}
In particular, $\sigma(0+) = 0$.
\end{lemma}

The next result is essentially Lemma 5.6 from \cite{GS}.

\begin{lemma}\label{L:nice}
There exist $\alpha>0$ such that
\begin{equation}\label{wltfm}
\left|\frac{\sigma(r)}{r} - \alpha\right| \le \beta e^\beta
r,\quad r\in (0,1),
\end{equation}
for $\beta$ as in Lemma~\ref{L:psi1}.
In particular,
\begin{equation}\label{wcifm}
\alpha=\lim_{r\to 0+} \frac{\sigma(r)}{r}.
\end{equation}
Moreover, we also have that $e^{-\beta}\leq\alpha\leq e^\beta$.
\end{lemma}

With $\psi$ as in \eqref{psi-sig-def}, we now define
\begin{equation}\label{Mii}
M_L(v,r) = \frac{1}{\psi(r)}H_L(v,r),\quad J_L(v,r) = \frac{1}{\psi(r)}I_L(v,r).
\end{equation}
The next relevant formulas are those of 
$J'(r)=\frac{d}{dr}J_L(v,r)$ and  $M'(r)=\frac{d}{dr}M_L(v,r)$.
Using formula (5.28) in \cite{GS}, we have
\begin{align}\label{J'Jnew}
J'(r) & = \left(- \frac{\psi'(r)}{\psi(r)} +\frac{n-2}{r}+O(1)\right)
        J(r)+ \frac{1}{\psi(r)} \bigg\{2\int_{S_r}\frac{\langle
        A\nabla v,\nu\rangle^2}{\mu}\\
&\qquad\mbox{}-\frac{2}{r}\int_{B_r}\langle Z,\nabla v\rangle f
  -\left(\frac{n-2}{r}+O(1)\right)\int_{B_r}v f +\int_{S_r} v
  f\bigg\},
\notag
\end{align} 
where the vector field $Z$ is given by 
\begin{equation}\label{Z}
Z = \frac{rA\nabla r}{\mu}=\frac{A(x)x}{\mu}.
\end{equation}
We also recall that (5.26) in \cite{GS} gives 
\begin{equation}\label{M'new}
M'(r) = 2 J(r).
\end{equation}
The central result in \cite{GS} is the following monotonicity formula.

\begin{theo}[Monotonicity of the truncated frequency]\label{T:lym} Let
  $v$ satisfy \eqref{c00i}--\eqref{c44i} with $f\in
  L^\infty(B_1)$. Given $\delta\in (0,1)$ there exist universal
  numbers $r_0, K'>0$, depending also on $\delta$ and
  $\|f\|_{L^\infty}$,  such that
the function
\begin{equation}\label{nNN}
N(r) = N_L(v,r) \defeq  \frac{\sigma(r)}{2} e^{K'
  r^{\frac{1-\delta}{2}}} \frac{d}{dr} \log \max
\left\{M_L(v,r),r^{3+\delta}\right\}.
\end{equation}
is monotone non-decreasing on $(0,r_0)$.
\end{theo}

We call $N(r)$ the \emph{truncated frequency function}, by analogy
with Almgren's frequency function \cite{A} (see \cites{GS,GP,CSS} for
more insights on this kind of formulas). 

We then define a modification of $N$ as follows:
\[
\tilde{N}(r)=\tilde{N}_L(v,r)\defeq \frac{r}{\sigma(r)}N_L(v,r) =
\frac{r}{2}e^{K'r^{\frac{1-\delta}{2}}}\frac{d}{dr}\log\max\{M_L(v,r),
r^{3+\delta}\}.
\]
We notice that by Theorem~\ref{T:lym} the limit $N(0+)$ exists. Combining
that with Lemma~\ref{L:nice} above, which states that
$\lim\limits_{r\rightarrow 0+}\frac{\sigma(r)}{r}=\alpha>0$, we see
that  $\tilde{N}(0+)$ also exists. 

The following lemma provides a summary of estimates which are crucial
for our further study. The lower bound on $\tilde{N}(0+)$ is proved in
Lemma 6.3 in \cite{GS} (whose proof contains also that of the gap
on the possible values of $\tilde{N}(0+)$), the bound on $|v(x)|$ is
Lemma 6.6 and the bound on $|\nabla v(x)|$ is proved in Theorem 6.7
there.

\begin{lemma}\label{L:newcrucial}
Let $v$ satisfy \eqref{c00i}--\eqref{c44i} with $f\in
  L^\infty(B_1)$, and let $r_0\in(0,1/2]$ be as in
  Theorem~\ref{T:lym}. Then, $\tilde{N}(0+)\ge \frac 32$, and actually
  $\tilde{N}(0+)=\frac 32$ or $\tilde{N}(0+)\ge \frac{3+\delta}{2}$.
  
  Moreover, there exists a universal $C$ depending also on  $\delta$,
  $H(r_0)$ and $\|f\|_{L^\infty(B_1)}$ such that
\begin{equation}\label{vbounds}
|v(x)|\le C |x|^{3/2},\quad |\nabla v(x)| \le C |x|^{1/2},\quad |x|\le r_0.
\end{equation}
\end{lemma}

\begin{cor}\label{C:HI}
With $r_0$ as in Theorem~\ref{T:lym}, one has
\begin{equation}\label{mf4}
H(r)\leq C r^{n+2},\quad |I(r)|\le C r^{n+1},\quad r\le r_0.
\end{equation}
\end{cor}

\begin{proof}
It is enough to use \eqref{vbounds} in definitions \eqref{Hv} and
\eqref{Iv} above.
\end{proof}

The results of this section have been stated when the free boundary
point in question is 
the origin. However, given any $x_0\in\Gamma(v)$, we can move $x_0$ to
the origin by letting 
\begin{align*}
v_{x_0}(x)&=v(x_0+A^{1/2}(x_0)x)-b_{x_0}x_n,\quad \text{where}\ 
b_{x_0} =\langle A^{1/2}(x_0)\nabla v(x_0), e_n\rangle,
\\
A_{x_0}(x)&=A^{-1/2}(x_0)A(x_0+A^{1/2}(x_0)x)A^{-1/2}(x_0),
\\
\mu_{x_0}(x)&=\langle A_{x_0}(x)\nu(x),\nu(x)\rangle,
\\
L_{x_0}& = \operatorname{div}(A_{x_0} \nabla \cdot).
\end{align*}
(Note that, by the $C^{1,\frac 12}$-regularity of $v$
established in \cite{GS}, the mapping $x_0\mapsto b_{x_0}$ is
$C^{\frac 12}$ on $\Gamma(v)$.) 
Then, by construction we have the normalizations
$A_{x_0}(0)=I_n$, $\mu_{x_0}(0)=1$. We also know that
$0\in\Gamma(v_{x_0})$, and that
\[
v_{x_0}(0)= v(x_0) = 0,\quad |\nabla v_{x_0}(0)|=0.
\]
Besides,
$v_{x_0}$ satisfies \eqref{c11i}--\eqref{c44i} for the operator
$L_{x_0}$. Thus, all results stated above for $v$ are also applicable
to $v_{x_0}$.

We thus also have the versions of the quantities defined in this
sections, such as $M_L$, $N_L$, etc, centered at $x_0$ (if we
replace $L$ with $L_{x_0}$). But instead of using the overly bulky
notations $M_{L_{x_0}}$, $N_{L_{x_0}}$, etc, we will use $M_{x_0}$,
$N_{x_0}$, etc.

\section{Regular free boundary points}\label{S:reg}

Using Theorem~\ref{T:lym}, in this section we explore in more detail
the notion of \emph{regular free boundary points} and establish some
preliminary properties of the regular set. 
We begin by recalling the following definition from Section \ref{S:intro}.

\begin{dfn}\label{D:reg}
We say that $x_0\in \Gamma(v)$ is \emph{regular} iff $\tilde
N_{x_0}(v_{x_0},0+) = \frac 32$
and let $\Gamma_{3/2}(v)$ be the set of all regular free
boundary points. $\Gamma_{3/2}(v)$  is also called the \emph{regular set}.
\end{dfn}

In Lemma~\ref{uniformN(r)3/2}  below we prove that $\Gamma_{3/2}(v)$
is a relatively open subset of the free boundary $\Gamma(v)$.
To accomplish this we prove that $\tilde{N}_{\bar{x}}(0+)=3/2$ for
$\bar{x}$ in a small neighborhood of $x_0\in \Gamma_{3/2}(v)$. Since
the definition of $\tilde{N}(r)$ involves a truncation of
$M(r)$, we first need to establish the following auxiliary result.

\begin{lemma}\label{L:boundonHr} Let $v$ satisfy
\eqref{c11i}--\eqref{c44i} with $0\in\Gamma_{3/2}(v)$. Then
$$
\frac{r}2\frac{M'(r)}{M(r)}\to \frac{3}{2}\quad\text{as }r\to 0+.
$$
In particular, for every $\ve>0$ there exists
$r_\ve>0$ and $C_\ve>0$ such that 
\begin{align*}
\frac{r}2\,\frac{M'(r)}{M(r)}\leq \frac{3+\ve}2,\quad M(r)\geq C_\ve
  r^{3+\ve},\quad\text{for }0<r\leq r_\ve.
\end{align*}
\end{lemma}

\begin{proof}
We first claim that since $0\in\Gamma_{3/2}(v)$, then $M(r)\ge
r^{3+\delta}$ for $r>0$ small. Indeed, if there was a sequence
$s_j\rightarrow 0$ such that $M(s_j)< s_j^{3+\delta}$, then
\[
\tilde{N}(s_j)=\frac{3+\delta}{2}e^{K's_j^{\frac{1-\delta}{2}}}\rightarrow
\frac{3+\delta}{2}\ne \frac{3}{2}=\tilde{N}(0+),
\]
which is a contradiction.
Hence, for $r$ small we have 
\[
\tilde{N}(r)=\frac{r}{2}e^{K'r^{\frac{1-\delta}{2}}}\frac{M'(r)}{M(r)}.
\]
Since $e^{K'r^{\frac{1-\delta}{2}}}\to 1$ as $r\to 0$ and
$\tilde{N}(0+)=3/2$, we obtain the first part of the lemma. Hence,
for every $\ve>0$ there exists a small $r_\ve>0$ such that
\[
r\frac{M'(r)}{M(r)}\leq 3+\ve,\quad r<r_\ve.
\]
Integrating from $r$ to $r_\ve$, this gives
\[
\frac{M(r_\ve)}{M(r)}\leq \left(\frac{r_\ve}{r}\right)^{3+\ve},
\]
from which we conclude, with $C_\ve = M(r_\ve)/r_\ve^{3+\ve}$, that
$M(r)\geq C_{\ve}r^{3+\ve}$.
\end{proof}

\begin{lemma}\label{uniformN(r)3/2} Let $v$ satisfy
  \eqref{c11i}--\eqref{c44i} with $x_0\in\Gamma_{3/2}(v)$. Then, there
  exists $\eta_0=\eta_0(x_0)>0$ such that  $\Gamma(v)\cap
  B_{\eta_0}'(x_0)\subset \Gamma_{3/2}(v)$ and, moreover, the 
  convergence 
$$
\tilde N_{\bar x}(r)\to 3/2\quad\text{as}\quad r\to 0+
$$
is uniform for $\bar x\in\Gamma(v)\cap\overline {B_{\eta_0/2}'(x_0)}$.
\end{lemma}
\begin{proof} Fix $0<\varepsilon<\delta/8$, where $\delta>0$ is fixed
  as in the definition \eqref{nNN} of the frequency
  $N_{x_0}(r)$. Then, by Lemma~\ref{L:boundonHr}, there exist
  $C_\ve = C_\ve(x_0)>0$ and $r_\varepsilon=r_\varepsilon(x_0)>0$ such that
$$
\frac{r}{2}\frac{M'_{x_0}(r)}{M_{x_0}(r)}\leq
\frac{3+\varepsilon}{2},\quad M_{x_0}(r)\geq C_\varepsilon r^{3+\varepsilon},\quad r<r_\varepsilon.
$$
We then want to show that similar inequalities will hold if we replace
$x_0$ with $\bar x\in
B'_{\eta_\varepsilon}(x_0)$ for a sufficiently small
$\eta_{\varepsilon}$. We will write $L_{\bar x}=\text{div}(A_{\bar
  x}\nabla \cdot)$. To track the dependence on $\bar x$,
we write, using the differentiation formulas in \cite{GS}, that
\begin{align*}
\frac{M'_{\bar x}(r)}{M_{\bar x}(r)}&=\frac{H'_{\bar x}(r)}{H_{\bar x}(r)}-\frac{\psi'_{\bar x}(r)}{\psi_{\bar x}(r)}=\frac{H'_{\bar x}(r)}{H_{\bar x}(r)}-\frac{\int_{S_r}
  v_{\bar x}(x)^2 L_{\bar x}|x|}{\int_{S_r}v_{\bar x}(x)^2\mu_{\bar x}(x)}\\
&=\frac{2 I_{\bar x}(r)}{H_{\bar x}(r)}=\frac{2\int_{S_r} v_{\bar x}(x)\langle
  A_{\bar x}(x)\nabla v_{\bar x}(x),\nu\rangle}{\int_{S_r}
  v_{\bar x}(x)^2\mu_{\bar x}(x)}\\
&=\frac{2r\int_{S_r} v_{\bar x}(x)\langle
  A_{\bar x}(x)\nabla v_{\bar x}(x),x\rangle}{\int_{S_r}
  v_{\bar x}(x)^2\langle A_{\bar x}(x)x,x\rangle}.
\end{align*}
This implies that for fixed $r>0$, the mapping $\bar x\mapsto
M_{\bar x}'(r)/M_{\bar x}(r)$ is continuous. Thus, if
$\rho_\varepsilon<r_\varepsilon$, we
can find a small $\eta_\varepsilon>0$ such that
$$
\frac{\rho_\varepsilon}{2}
\frac{M_{\bar x}'(\rho_\varepsilon)}{M_{\bar x}(\rho_\varepsilon)}\leq
\frac{3+2\varepsilon}{2},\quad\text{for }\bar x\in B'_{\eta_\varepsilon}(x_0).
$$ 
On the other hand, since
$$
M_{\bar x}(r)=\frac{r}{\sigma_{\bar x}(r)} \frac{1}{r^{n+1}}\int_{S_r}
v_{\bar x}(x)^2 \langle A_{\bar x}(x) x, x\rangle,
$$
the continuity of the mapping $\bar x \to M_{\bar x}(r)$
is not so clear. However, having that $c_0\leq \frac{\sigma_{\bar
  x}(r)}{r} \leq c_0^{-1}$ for a universal constant $c_0>0$, we can write that
$$
M_{\bar x}(\rho_\varepsilon)\geq
\frac12c_0^2C_\varepsilon\rho_{\varepsilon}^{3+\varepsilon}>
\rho_\varepsilon^{3+\delta},\quad \bar x\in B'_{\eta_\varepsilon}(x_0),
$$
if we take $\eta_\varepsilon$ and $\rho_\varepsilon$ sufficiently
small. The latter inequality
implies that we can explicitly compute $\tilde N_{\bar
  x}(\rho_\varepsilon)$ by
$$
\tilde N_{\bar
  x}(\rho_\varepsilon)=\frac{\rho_\varepsilon}{2}
e^{K'\rho_\varepsilon^{(1-\delta)/2}}\frac{M_{\bar
    x}'(\rho_\varepsilon)}{M_{\bar x}(\rho_\varepsilon)}\le
e^{K'\rho_\varepsilon^{(1-\delta)/2}}\frac{3+2\varepsilon}{2} \leq
\frac{3+3\varepsilon}{2},
$$
again if $\rho_\varepsilon$ is small enough. Hence, by the monotonicity
of $N_{\bar x}(r)=(\sigma_{\bar x}(r)/r)\tilde N_{\bar x}(r)$, we obtain 
$$
\tilde N_{\bar x} (0+)\leq \frac{1}{\alpha_{\bar x}}
\frac{\sigma_{\bar x}(\rho_\varepsilon)}{\rho_\varepsilon}
\frac{3+3\varepsilon}{2},
$$
where 
\begin{equation}\label{alphaxbar}
\alpha_{\bar x} \defeq  \lim_{r\to 0+}\frac{\sigma_{\bar x}(r)}r.
\end{equation}
Using now the estimates in Lemma~\ref{L:nice}, we have
$$
\left|\frac{\sigma_{\bar x}(r)}{r}-\alpha_{\bar x}\right|\leq C_0
r\quad \text{and} \quad \alpha_{\bar x} \ge c_0,
$$
therefore we can guarantee that
$$
\tilde N_{\bar x}(0+)\leq (1+\varepsilon)\frac{3+3\varepsilon}2\leq
\frac{3+7\varepsilon}{2}<\frac{3+\delta}2,\quad \bar x\in
B'_{\eta_\varepsilon}(x_0).
$$
But then, by the gap of values of $\tilde N_{\bar x}(0+)$ between $3/2$ and
$(3+\delta)/2$, we conclude that
\[
\tilde N_{\bar x}(0+)=3/2,\quad \bar x\in B'_{\eta_\varepsilon}(x_0).
\]
To prove the second part of the lemma, we note that for any fixed
$\bar x\in B_{\eta_\varepsilon}'(x_0)$, the mapping
$$
r\mapsto e^{\beta r} \tilde N_{\bar x}(r)\quad 0<r<\rho_\varepsilon
$$
is monotone increasing for a universal constant $\beta>0$, which
follows from the inequality 
$$
\left|\frac{d}{dr}\log\frac{\sigma_{\bar x}(r)}{r}\right|\leq
\beta,\qquad \text{(see Lemma~\ref{L:nice})},
$$
the monotonicity of $r\mapsto N_{\bar x}(r)=\frac{\sigma_{\bar
    x}(r)}{r}\tilde N_{\bar x}(r)$, as well as the nonnegativity of
$\tilde N_{\bar x}(r)$.

Now, for each fixed $0<r<\rho_\varepsilon$, the mapping
$$
\bar x\mapsto e^{\beta r} \tilde{N}_{\bar x}(r)=e^{\beta r+ K'
  r^{(1-\delta)/2}}\frac{r^2\int_{S_r} v_{\bar x}(x)\langle
  A_{\bar x}(x)\nabla v_{\bar x}(x),x\rangle}{\int_{S_r}
  v_{\bar x}(x)^2\langle A_{\bar x}(x)x,x\rangle}
$$
is continuous on $B'_{\eta_\varepsilon}(x_0)$. Since the limit 
$$
\lim_{r\to 0+}e^{\beta r} \tilde N_{\bar
  x}(r)=\tilde N_{\bar x}(0+)=3/2,\quad\text{for all } \bar x\in
B'_{\eta_\varepsilon}(x_0),
$$
by the classical theorem of Dini, we have that the convergence
$e^{\beta r}\tilde N_{\bar x}(r)\to 3/2$ as $r\to 0+$ will be uniform on
$\Gamma(v)\cap\overline{B'_{\eta_\varepsilon/2}(x_0)}$, implying also
the uniform
convergence $\tilde N_{\bar x}(r)\to 3/2$.
\end{proof}

In the remaining part of this section we study  \emph{Almgren type scalings}
\begin{equation}\label{Almgrentype}
\tilde v_{\bar x, r}(x)=\frac{v_{\bar x}(rx)}{d_{\bar x,r}},\quad
d_{\bar x, r}= \left(\frac{1}{r^{n-1}}H_{\bar x}(r)\right)^{1/2}.
\end{equation}
This is slightly different from what was done in \cite{GS}, but more
suited for the study of the free boundary.
Notice that we have the following normalization:
$$
\int_{S_1} \tilde v_{\bar x,r}^2\mu_{\bar x, r}=1,\quad \mu_{\bar x,
  r}=\frac{\langle A_{\bar x} (rx) x,x\rangle}{|x|^2}.
$$
Now, if $\bar x\in \Gamma_{3/2}(v)$, then the results of \cite{GS}
imply that over subsequences $r=r_j\to 0$, we have the convergence 
\[
\tilde{v}_{\bar{x},r}(x)\left(\frac{\sigma_{\bar{x}}(r)}{r}\right)^{1/2}
\rightarrow a\Re(\langle x',e'\rangle +i|x_n|)^{3/2}\quad
\text{ in }C^{1,\alpha}_{\text{loc}}(\R^n_\pm\cup \R^{n-1})
\]
for some $e'\in\R^{n-1}$ with $|e'|=1$ and $a>0$. Since we also have
the convergence
\[
\frac{\sigma_{\bar{x}}(r)}{r}\rightarrow \alpha_{\bar{x}}>0,
\]
see Lemma~\ref{L:nice}, we obtain that
\[
\tilde{v}_{\bar{x},r}(x)\rightarrow
\frac{a}{\alpha_{\bar{x}}^{1/2}}\Re(\langle x',e'\rangle
+i|x_n|)^{3/2}\quad\text{ in }C^{1,\alpha}_{\text{loc}}(\R^n_\pm\cup \R^{n-1}).
\]
The normalization
$\int_{S_1}\tilde{v}_{\bar{x},r}^2\mu_{\bar{x},r}=1$ then implies
that there exists a dimensional constant $c_n>0$ (independent of $\bar
x$), such that, on a subsequence,
$$
\tilde v_{\bar x, r}(x)\to c_n \Re (\langle x', e'\rangle +i |x_n|)^{3/2}.
$$
Moreover, we can actually prove the following result.

\begin{lemma}\label{vxr-h} Let $v$ satisfy \eqref{c11i}--\eqref{c44i} with
  $x_0\in\Gamma_{3/2}(v)$. Given $\theta>0$, there exists $r_0=r_0(x_0)>0$ and
  $\eta_0=\eta_0(x_0)>0$ such that
$$
\inf_{e'\in\R^{n-1}, |e'|=1}\|\tilde v_{\bar x,r}(x)-c_n\Re(\langle
x', e'\rangle +i|x_n|)^{3/2}\|_{C^{1,\alpha}(B_1^{\pm}\cup B_1')}<\theta
$$
for any $\bar x\in=\Gamma_{3/2}(v)\cap\overline{B_{\eta_0}'(x_0)}$ and $r<r_0$.
\end{lemma}

\begin{proof} Suppose the contrary. Then, there exists a sequence
  $\bar x_j\to x_0$ and $r_j\to 0$ such that
$$
\|\tilde v_{\bar x_j, r_j}-c_n\Re(\langle x', e'\rangle
+i|x_n|)^{3/2}\|_{C^{1,\alpha}(B_1^{\pm}\cup B_1')}\geq \theta
$$
for any unit vector $e'\in\R^{n-1}$. Observe now that
\begin{equation}\label{rescaleN(r)}
\begin{aligned}
e^{-K'(\rho r)^{(1-\delta)/2}}\tilde N_{\bar x}(\rho r)&=\frac{(\rho
  r)^2\int_{S_{\rho r}} v_{\bar x}(x)\langle
  A_{\bar x}(x)\nabla v_{\bar x}(x),x\rangle}{\int_{S_{\rho r}}
  v_{\bar x}(x)^2\langle A_{\bar x}(x)x,x\rangle}\\
&=\frac{\rho^2\int_{S_\rho} \tilde v_{\bar x,r}(x)\langle
  A_{\bar x}(rx)\nabla \tilde v_{\bar x,r}(x),x\rangle}{\int_{S_\rho}
  \tilde v_{\bar x, r}(x)^2\langle A_{\bar x}(rx)x,x\rangle}
\end{aligned}
\end{equation}
Now, we claim that the scalings $\tilde v_{\bar x_j, r_j}$ are uniformly
bounded in $C^{1,1/2}(B_R^\pm\cup B_R')$ for any $R>0$.
Indeed, by Lemma~\ref{uniformN(r)3/2}, given $\varepsilon>0$ small, we have that
$$
\frac{t}{2}\frac{M'_{\bar x}(t)}{M_{\bar x}(t)}\leq
\frac{3+\varepsilon}{2},\quad t<\rho_\varepsilon,\ \bar x \in
B_{\eta_\varepsilon}'(x_0).
$$
Let $R\geq 1$ and $Rr<\rho_\varepsilon$. Integrating the above
inequality from $t=r$ to $Rr$, we obtain that
$$
M_{\bar x}(Rr)\leq M_{\bar x}(r) R^{3+\varepsilon}.
$$
Changing $M_{\bar x}$ to $H_{\bar x}$ we therefore have, using
\eqref{per5}, that
$$
H_{\bar x}(Rr)\leq C_0 H_{\bar x}(r) R^{n+2+\varepsilon}.
$$
The latter can we written in the form
$$
\int_{S_R} \tilde v_{\bar x, r}^2\mu_{\bar x, r}\leq C_0 R^{n+2+\varepsilon}.
$$
Thus, the uniform boundedness of $\tilde v_{\bar x_j, r_j}$ in
$L^2(B_R)$, and consequently in $C^{1,1/2}(B_{R/2}^\pm\cup B_{R/2}')$,
follows. 
Hence, we can assume without loss of generality that $\tilde v_{\bar x_j,
  r_j}\to v_0$ in $C^{1,\alpha}_\loc(\R^n_\pm\cup\R^{n-1})$. It is
immediate to see that $v_0$ will solve the Signorini problem for the
Laplacian in the entire $\R^n$. Besides, by
Lemma~\ref{uniformN(r)3/2}, we will have that $\tilde N_{\bar
  x_j}(\rho r_j)\to 3/2$. On the other hand, passing to the limit in
\eqref{rescaleN(r)}, we obtain that
$$
\frac32=\frac{\int_{S_\rho} v_{0}(x)\langle
  \nabla v_{0}(x),x\rangle}{\int_{S_\rho}
  v_{0}(x)^2}=\frac{\rho\int_{B_\rho} |\nabla v_0|^2}{\int_{S_\rho}
  v_{0}(x)^2},\quad\text{for any }\rho>0.
$$ 
Therefore, $v_0$ is $3/2$ homogeneous global solution of the Signorini
problem and thus has the form $v_0(x)=c_n\Re(\langle x',
e_0'\rangle +i|x_n|)^{3/2}$. Hence, for large $j$ we will have
$$
\|\tilde v_{\bar x_j, r_j}-c_n\Re(\langle x',
e_0'\rangle +i|x_n|)^{3/2}\|_{C^1(B_1^\pm\cup B_1')}< \theta,
$$
contradictory to our assumption. The proof is complete.
\end{proof}

\section{A Weiss type monotonicity formula}\label{S:weiss}

In this section we establish a monotonicity formula which is
reminiscent of that established by Weiss in \cite{W} for the classical
obstacle problem, and which is one of the two main ingredients in our
proof of the $C^{1,\beta}$ regularity of the regular set. We consider
the solution $u$ to the Signorini
problem \eqref{bb1i}--\eqref{bb4i} above, and we set $v$ as in
\eqref{vi}.

\begin{dfn}\label{D:weiss}
Let $r_0>0$ be as in Theorem~\ref{T:lym}. For $r\in (0,r_0)$ we define
the \emph{$\frac 32$-th generalized Weiss-type functional} as follows
\begin{align}\label{W}
W_L(v,r) &= \frac{\sigma(r)}{r^{3}} \left\{J_L(v,r) - \frac{3/2}{r}
  M_L(v,r)\right\}\\
&=\frac{1}{r^{n+1}}I_L(v,r)-\frac{3/2}{r^{n+2}}H_L(v, r)\label{WW'}\\
&=\frac{1}{r^{n+1}}\int_{B_r}[\langle A(x)\nabla v, \nabla
v\rangle +vf] - \frac{3/2}{r^{n+2}}\int_{S_r}v^2\mu,\nonumber
\end{align}
where $I_L$, $J_L$, $H_L$, and $M_L(v,r)$ are as in
Section~\ref{S:not}. Whenever $L=\Delta$, we write
$W(v,r)=W_{\Delta}(v,r)$, and unless we want to stress the dependence
on $v$, we will write $W_L(r)=W_L(v,r)$.
\end{dfn}

In this section we will show that there exists $C>0$ such that
$r\mapsto W_L(v,r)+Cr^{1/2}$ is monotone nondecreasing, that the limit
$\lim\limits_{r\rightarrow 0}W_L(v,r)$ exists and is zero, and that
$W_L(v,r)\ge -Cr^{1,2}.$ We start by proving that $W_L(v,\cdot)$ is
bounded.

\begin{lemma}\label{L:Wbounded}
The functional $W_L(v,\cdot)$ is bounded on the interval $(0,r_0)$.
\end{lemma}

\begin{proof}
Indeed, from Corollary \ref{C:HI} we have
\[
|W_L(v,r)| 
\le \frac{1}{r^{n+1}} \left\{|I(r)| + \frac{3}{2r} H(r)\right\}\le C.\qedhere
\]
\end{proof}

The functional $W_L(r)$ in \eqref{W} is tailor-made to the study of
regular free boundary points of solutions of the Signorini problem
\eqref{c11i}--\eqref{c44i}. The following ``almost monotonicity''
property of $W_L$ plays a crucial role in our further study.

\begin{theo}[Weiss type monotonicity formula]\label{T:weiss}
Let $v$ satisfy \eqref{c11i}--\eqref{c44i} and let $0\in
\Gamma_{3/2}(v)$. Then, there exist universal constant $C, r_0>0$,
depending also on $\|f\|_{L^\infty(B_1)}$, such that for every
$0<r<r_0$ one has
\begin{equation}\label{W'}
\frac{d}{dr}\left(W_L(v,r) + C r^{1/2}\right) \ge \frac{2}{r^{n+1}}
\int_{S_r}\bigg(\frac{\langle A\nabla v,\nu\rangle }{\sqrt \mu} -
\frac{(3/2)\sqrt \mu}{r} v\bigg)^2.
\end{equation}
In particular, there exists $C>0$ such that function $r\mapsto
W_L(v,r)+Cr^{1/2}$ is monotone non\-de\-creasing, and therefore the
limit $W_L(v,0+) \defeq  \lim\limits_{r\rightarrow 0}W_L(v,r)$
exists.
\end{theo}

\begin{proof}
We have from Definition \ref{D:weiss} above,
\begin{align*}
\frac{d}{dr}W_L(v,r) & = \frac{\sigma(r)}{r^{3}}\left\{\left(\frac{\sigma'(r)}{\sigma(r)} - \frac{3}{r}\right)\left(J(r) - \frac{3/2}{r} M(r)\right)  + \left(J'(r) - \frac{3/2}{r} M'(r) + \frac{3/2}{r^2} M(r)\right)\right\}
\\
& =  \frac{\sigma(r)}{r^{3}}\left\{\left(\frac{\psi'(r)}{\psi(r)} - \frac{n-2}{r} - \frac{3}{r}\right) \left(J(r) - \frac{3/2}{r} M(r)\right) + \left(J'(r) - \frac{3/2}{r} M'(r) + \frac{3/2}{r^2} M(r)\right)\right\}.
\end{align*}
Using \eqref{J'Jnew} and \eqref{M'new} above, we thus find
\begin{align*}
\frac{d}{dr}W_L(v,r) & = \frac{\sigma(r)}{r^{3}}\bigg\{\left(\frac{\psi'(r)}{\psi(r)} - \frac{n-2}{r} - \frac{3}{r}\right)\left(J(r) - \frac{3/2}{r} M(r)\right) 
\\ 
&\qquad\mbox{} + \left(\bigg(- \frac{\psi'(r)}{\psi(r)} +\frac{n-2}{r}+O(1)\right) J(r)+ \frac{1}{\psi(r)} \bigg\{2\int_{S_r}\frac{\langle A\nabla v,\nu\rangle^2}{\mu}   
\\
&\qquad\mbox{} -\frac{2}{r}\int_{B_r}\langle Z,\nabla v\rangle f -\left(\frac{n-2}{r}+O(1)\right)\int_{B_r}v f +\int_{S_r} v f\bigg\} - \frac{3}{r} J(r) + \frac{3/2}{r^2} M(r)\bigg)\bigg\}
\\
& = \frac{\sigma(r)}{r^{3}}\bigg\{\left(- \frac{6}{r} + O(1)\right) J(r) + \frac{2}{\psi(r)} \int_{S_r}\frac{\langle A\nabla v,\nu\rangle^2}{\mu} 
\\
&\qquad\mbox{} + \frac{3/2}{r^2}\left(1 - \left(\frac{r\psi'(r)}{\psi(r)} - n+2 - 3\right)\right) M(r) 
\\
&\qquad\mbox{} + \frac{1}{\psi(r)}\bigg(-\frac{2}{r}\int_{B_r}\langle Z,\nabla v\rangle f -\left(\frac{n-2}{r}+O(1)\right)\int_{B_r}v f +\int_{S_r} v f\bigg)\bigg\}. 
\end{align*}
By \eqref{per4} in Lemma~\ref{L:psi1} above, we see that $\frac{\psi'(r)}{\psi(r)} = \frac{n-1}{r} + O(1)$. We thus obtain from the latter chain of equalities
\begin{align*}
\frac{d}{dr}W_L(v,r) & = \frac{2\sigma(r)}{r^{3}\psi(r)}\bigg\{\left(- \frac{3}{r} + O(1)\right) I(r) +\int_{S_r}\frac{\langle A\nabla v,\nu\rangle^2}{\mu}  + \frac{9/4}{r^2}\big(1 + O(r)\big) H(r)\bigg\}
\\
&\qquad\mbox{} + \frac{\sigma(r)}{r^{3}\psi(r)}\bigg(-\frac{2}{r}\int_{B_r}\langle Z,\nabla v\rangle f -\left(\frac{n-2}{r}+O(1)\right)\int_{B_r}v f +\int_{S_r} v f\bigg).
\end{align*}
By the definitions \eqref{Hv} and \eqref{Iv} of $H(r)$ and $I(r)$ we have
\[
\int_{S_r}\bigg(\frac{\langle A\nabla v,\nu\rangle}{\sqrt \mu} - \frac{(3/2)\sqrt \mu}{r} v\bigg)^2 = \int_{S_r}\frac{\langle A\nabla v,\nu\rangle ^2}{\mu} - \frac{3}{r} I(r) + \frac{9/4}{r^2} H(r).
\]
Since $\sigma(r)=r^{2-n}\psi(r)$, we conclude that
\begin{align}\label{mf}
\frac{d}{dr}W_L(v,r) & = \frac{2}{r^{n+1}} \bigg\{\int_{S_r}\bigg(\frac{\langle A\nabla v,\nu\rangle }{\sqrt \mu} - \frac{(3/2)\sqrt \mu}{r} v\bigg)^2 + O(1) I(r) + \frac{O(1)}{r} H(r)\bigg\}
\\
&\qquad\mbox{} + \frac{1}{r^{n+1}}\bigg(-\frac{2}{r}\int_{B_r}\langle Z,\nabla v\rangle f -\left(\frac{n-2}{r}+O(1)\right)\int_{B_r}v f +\int_{S_r} v f\bigg).
\notag
\end{align}
Returning to \eqref{mf} and making use of  \eqref{mf4}, we conclude that
\begin{align}\label{mf5}
\frac{d}{dr}W_L(v,r)  & = \frac{2}{r^{n+1}} \int_{S_r}\bigg(\frac{\langle A\nabla v,\nu\rangle }{\sqrt \mu} - \frac{(3/2)\sqrt \mu}{r} v\bigg)^2 + O(1) 
\\
&\qquad\mbox{} + \frac{1}{r^{n+1}}\bigg(-\frac{2}{r}\int_{B_r}\langle Z,\nabla v\rangle f -\left(\frac{n-2}{r}+O(1)\right)\int_{B_r}v f +\int_{S_r} v f\bigg).
\notag
\end{align}
The proof of the estimate \eqref{W'}
will be completed if we can show that there exists a universal
constant $C>0$ such that 
\[
\bigg|-\frac{2}{r}\int_{B_r}\langle Z,\nabla v\rangle f -\left(\frac{n-2}{r}+O(1)\right)\int_{B_r}v f +\int_{S_r} v f\bigg| \le C r^{n+\frac 12}.
\]
From the expression of the vector field $Z=\frac{A(x)x}{\mu}$, see
\eqref{Z} above, we have $|Z|\le C r$ for $|x|\le r$. Since $f\in
L^\infty$, we obtain from the second inequality in \eqref{vbounds} 
\[
\left|-\frac{2}{r}\int_{B_r}\langle Z,\nabla v\rangle f\right| \le C
r^{n+\frac 12},
\]
for a universal $C>0$ which also depends on $\|f\|_{L^\infty(B_1)}$. The first inequality in \eqref{vbounds} gives instead
\[
\left|\left(\frac{n-2}{r}+O(1)\right)\int_{B_r}v f +\int_{S_r} v f\right| \le C r^{n+\frac 12},
\]
thus completing the proof of \eqref{W'}.
The existence of $W_L(v,0+)$ now follows from the monotonicity and
boundedness of $W_L(v,r)+Cr^{1/2}$, see Lemma~\ref{L:Wbounded}.
\end{proof}

From Theorem~\ref{T:weiss} we obtain that $W_L(v,
0+)=\lim\limits_{r\rightarrow 0}W_L(v,r)$ exists. In the next lemma we
prove that this limit must actually be zero.

\begin{lemma}\label{L:W0nonzero}  Let $v$ satisfy
  \eqref{c00i}--\eqref{c44i} with $f\in L^\infty(B_1)$ and  $0\in
  \Gamma_{3/2}(v)$. Then,  $W_L(v,0+) = 0$.
\end{lemma}

\begin{proof} Recall from \eqref{W} that one has
\begin{equation}\label{W322}
W_L(v,r)= W_L(r) = \frac{\sigma(r)}{2r^{3}}\left\{2J(r)-\frac{3}{r}M(r)\right\} = \frac{H(r)}{r^{n+2}}\left\{\frac{r}2\,\frac{M'(r)}{M(r)}- \frac32\right\},
\end{equation}
where in the last equality we have used \eqref{M'new} and \eqref{Mii}
above. The proof now follows from the boundedness of $\frac{H(r)}{r^{n+2}}$ by
\eqref{mf4} and the convergence
$\frac{r}2\,\frac{M'(r)}{M(r)}\to\frac32$ by Lemma~\ref{L:boundonHr}.
\end{proof}

\begin{cor}\label{C:weiss}
Let $C$ and $r_0$ be as in Theorem~\ref{T:weiss}. Then, for every
$0<r<r_0$ one has
\[
W_L(v,r) \ge - C r^{1/2}.
\]
\end{cor}

\begin{proof}
This follows directly by combining Theorem~\ref{T:weiss} with
Lemma~\ref{L:W0nonzero}.
\end{proof}

\section{Homogeneous blowups}\label{S:hbu}

In this section we analyze the uniform limits of some appropriate
scalings of a solution $v$ to the Signorini problem
\eqref{c11i}--\eqref{c44i} by making essential use of the monotonicity
formula in Theorem~\ref{T:weiss} above. These scalings, together with
the Almgren type ones defined in \eqref{Almgrentype}, will be
instrumental in Section \ref{S:final}.

Let $v$ satisfy \eqref{c11i}--\eqref{c44i} and let $0\in
\Gamma_{3/2}(v)$. We consider the following \emph{homogeneous
  scalings} of $v$
\begin{equation}\label{s}
v_r(x)=\frac{v(rx)}{r^{3/2}}.
\end{equation}

\begin{lemma}\label{L:eqvr} Define $A_r(x)=A(rx)$ and
  $f_r(x)=r^{1/2}f(rx)$. Then, $v_r$ solves the thin obstacle problem
  \eqref{c11i}--\eqref{c55i} in $B_{\frac{1}{r}}$ with zero thin
  obstacle, operator $L_r=\div(A_r\nabla\cdot)$ and right-hand side
  $f_r$ in \eqref{c11i}.
\end{lemma}

\begin{proof}
We only need to verify \eqref{c55i}. Indeed, given $\eta\in
C^{\infty}_0(B_{\frac{1}{r}})$, define $\rho\in C^\infty_0(B_1)$ by
letting $\rho(y)=r^{1/2}\eta\left(\frac{y}{r}\right)$. By a change of
variable one easily verifies that
\begin{align}\label{mf6}
\int_{B_{\frac{1}{r}}}\langle A_r\nabla v_r,\nabla \eta\rangle &= \int_{B_{\frac{1}{r}}'}\left[\langle A_r(x)\nabla v_r(x),\nu_+\rangle +\langle A_r(x)\nabla v_r(x),\nu_-\rangle\right]\eta(x)
\\
&\qquad\mbox{} - \int_{B_{\frac{1}{r}}}f_r(x)\eta(x).
\notag\qedhere
\end{align}
\end{proof}

\begin{lemma}\label{L:homblowupnonzero} Let $v$ satisfy
  \eqref{c11i}--\eqref{c44i} and let $0\in \Gamma_{3/2}(v)$. Given
  $r_j\rightarrow 0$, there exists a subsequence (which we will still
  denote by $r_j$) and a function $v_0\in
  C^{1,\alpha}_{\text{loc}}(\R^n_\pm\cup\R^{n-1})$
  for any $\alpha\in (0,1/2)$, such that $v_{r_j}\rightarrow v_0
  \text{ in } C^{1,\alpha}_\loc(\R^n_\pm\cup\R^{n-1})$. Such $v_0$ is
  a global solution of the Signorini
  problem \eqref{c11i}--\eqref{c44i} in $\Rn$  with zero thin obstacle and zero
  right-hand side $f$. 
\end{lemma}

\begin{proof} By Lemma~\ref{L:newcrucial}, there exist universal
  constants $C, r_0>0$ such that
\[
|v(x)| \le C|x|^{3/2}\quad\text{and }\quad |\nabla v(x)|\le
C|x|^{1/2},\quad |x|<r_0.
\]
 Moreover, as proved in \cite{GS}, $v\in C^{1,\frac
   12}_{\text{loc}}(B_1^{\pm}\cup B_1')$ with
\[
\|v\|_{C^{1,\frac{1}{2}}(B_{\frac{1}{2}}^{\pm}\cup B_{\frac{1}{2}}')}\leq C(n,\lambda,Q,\|v\|_{W^{1,2}(B_1)},\|f\|_{L^{\infty}}).
\]
Given $r_j\searrow 0$, by a standard diagonal process we obtain
convergence in $C^{1,\alpha}_{\text{loc}}(\R^n_\pm\cup\R^{n-1})$, for
any $\alpha \in (0,1/2)$, of a subsequence of the functions $v_{r_j}$
to a function $v_0$. Passing to the limit in \eqref{mf6} we conclude
that $v_0$ is a global solution to the Signorini problem with zero
thin obstacle. The fact (important for our later purposes) that the
right-hand side $f$ is also zero, follows again from \eqref{mf6} since
we obtain
\[
\left|\int_{B_{\frac{1}{r}}}f_r(x)\eta(x)\right| = \sqrt r \left|\int_{B_{\frac{1}{r}}}f(rx)\eta(x)\right| \le \|f\|_{L^\infty(B_1)}\|\eta\|_{L^1(\Rn)} \sqrt r \longrightarrow 0. \qedhere
\]
\end{proof}

\begin{rmrk}
Notice that we have not yet ruled out the possibility that $v_0\equiv
0$. This crucial aspect will be dealt with later, in Proposition
\ref{P:nondege} below. The next result establishes an important
homogeneity property of the global solution $v_0$.
\end{rmrk}

\begin{prop}\label{P:homblowupnonzero} Let $v_0$ be a function as in
  Lemma~\ref{L:homblowupnonzero}. Then, $v_0$ is homogeneous of degree
  $3/2$.
\end{prop}

\begin{proof} 
In what follows we denote with $\mu_{r}(x)=\mu(rx)$. Furthermore,
recall that the scaled functions in \eqref{s} are given by $v_r(x) =
r^{-3/2}\ v(rx)$. Let $r_j\searrow 0$ and denote by $v_0$ a
corresponding blowup as in Lemma~\ref{L:homblowupnonzero}.
Let $0<r<R$. We integrate \eqref{W'} over the interval $(r,R)$ obtaining  
\begin{align*}
W_L(v,r_jR)&-W_L(v,r_jr)-C\sqrt{r_j}(\sqrt{R}-\sqrt{r}) \ge \int_{r_jr}^{r_jR}\frac{2}{t^{n+3}}\int_{S_t}\left(\frac{\langle A\nabla v,x\rangle }{\sqrt{\mu}}-\frac{3}{2}\sqrt{\mu}v\right)^2 dt\\
&=\int_r^R\frac{1}{(r_js)^{n+3}}\int_{S_{r_js}}\left(\frac{\langle A\nabla v,x\rangle }{\sqrt{\mu}}-\frac{3}{2}\sqrt{\mu}v\right)^2 r_j ds\\
&=\int_r^R\frac{1}{(r_js)^{n+3}}\int_{S_s}\left(\frac{\langle A(r_jy)\nabla v(r_jy),r_jy\rangle }{\sqrt{\mu(r_jy)}}-\frac{3}{2}\sqrt{\mu(r_jy)}v(r_jy)\right)^2r_j^{n-1}r_j ds\\
&=\frac{1}{r_j^{3}}\int_r^R\frac{1}{s^{n+3}}\int_{S_s}\left(\frac{\langle A_{r_j}(y)\nabla v(r_jy),y\rangle r_j}{\sqrt{\mu_{r_j}(y)}}-\frac{3}{2}\sqrt{\mu_{r_j}(y)}v(r_jy)\right)^2 ds\\
&=\frac{1}{r_j^{3}}\int_r^R\frac{1}{s^{n+3}}\int_{S_s}\left(\frac{\langle A_{r_j}(y)r_j^{1/2}\nabla v_{r_j}(y),y\rangle r_j}{\sqrt{\mu_{r_j}(y)}}-\frac{3}{2}\sqrt{\mu_{r_j}(y)}v_{r_j}(y)r^{3/2}\right)^2 ds\\
&=\int_r^R\frac{1}{s^{n+3}}\int_{S_s}\left(\frac{\langle A_{r_j}(y)\nabla v_{r_j}(y),y\rangle }{\sqrt{\mu_{r_j}(y)}}-\frac{3}{2}\sqrt{\mu_{r_j}(y)}v_{r_j}(y)\right)^2 ds.
\end{align*}
Since $W_L(v,0+)$ exists, the left-hand side goes to zero as
$j\rightarrow \infty$. Since $A(0)=I, \mu(0)=1$ and we have
$C^{1,\alpha}_\loc(\R^n_\pm\cup\R^{n-1})$ convergence of $v_{r_j}$ to
$v_0$, passing to the limit as $j\to \infty$ in the above, we conclude
that
\[
0 \ge \int_r^R\frac{1}{s^{n+3}}\int_{S_s}\left(\langle\nabla v_0(y),y\rangle -\frac{3}{2} v_0(y)\right)^2ds.
\]
This inequality, and the arbitrariness of $0<r<R$, imply that $v_0$ is
homogeneous of degree $3/2$.
\end{proof}

\begin{dfn}\label{dfn:homog} We call such $v_0$ a \emph{homogeneous blowup}.
\end{dfn}

\section{An epiperimetric inequality for the Signorini problem}\label{S:epi}

In this section we establish in the context of the Signorini problem a
basic generalization  of the epiperimetric inequality obtained by
Weiss for the classical obstacle problem. Our main result, which is
Theorem~\ref{T:epi} below, is tailor made for analyzing regular free
boundary points in the Signorini problem, being the second main tool
we use to reach our goal -- the $C^{1,\beta}$ regularity of the
regular set.

\begin{dfn}\label{D:bae}
 Given $v\in W^{1,2}(B_1)$, we define the \emph{boundary adjusted
   energy} as the Weiss type functional defined in \eqref{W} for the
 Laplacian operator with $r=1$ and zero thin obstacle, i.e.,
\[
W(v)\defeq W_{\Delta}(v,1)=\int_{B_1}|\nabla v|^2-\frac{3}{2}\int_{S_1}v^2.
\]
\end{dfn}

\begin{rmrk}\label{R:weiss}
We observe explicitly that if $\int_{S_1} v^2 \not= 0$, then we can write
\[
W(v) = \left(\int_{S_1} v^2\right)\left[N(v,1) - \frac 32\right].
\]
It follows that if $v$ is a solution to the Signorini for the
Laplacian in $\Rn$, with zero thin obstacle, and which is homogeneous
of degree $\frac 32$, then by \cite{ACS} we have $N(v,r) \equiv \frac
32$, and therefore $W(v) = 0$.
\end{rmrk}

We now consider the function $h(x)=\Re(x_1+i|x_n|)^{\frac{3}{2}}$,
which is a $\frac 32$-homogeneous global solution of the Signorini
problem for the Laplacian with zero thin obstacle, and introduce the
set 
\[
H=\{ a\Re(\langle x',\nu\rangle+i|x_n|)^{\frac{3}{2}}\mid \nu\in S_1, a\ge 0\}
\]
of all multiples and rotations of the function $h$.
The following is the central result of this section.

\begin{theo}[Epiperimetric inequality]\label{T:epi}  There exists
  $\kappa\in (0,1)$ and $\theta\in (0,1)$ such that if $w\in
  W^{1,2}(B_1)$ is a homogeneous function of degree $\frac{3}{2}$ such
  that $w\ge 0$ on $B_1'$ and $\|w-h\|_{W^{1,2}(B_1)}\le \theta$, then
  there exists $\zeta\in W^{1,2}(B_1)$ such that $\zeta=w$ on $S_1$,
  $\zeta\ge 0$ on $B_1'$ and 
\[
W(\zeta)\le (1-\kappa)W(w).
\]
\end{theo}

\begin{proof}
We argue by contradiction and assume that the result does not
hold. Then, there exist sequences of real numbers $\kappa_m\rightarrow
0$ and $\theta_m\rightarrow 0$, and functions $w_m\in W^{1,2}(B_1)$,
homogeneous of degree $\frac{3}{2}$, such that $w_m\ge 0$ on $B_1'$
and
\begin{equation}\label{ate0}
\|w_m-h\|_{W^{1,2}(B_1)}\le \theta_m,
\end{equation}
but such that for every  $\zeta\in W^{1,2}(B_1)$ with $\zeta\ge 0$ on
$B_1'$, and for which $\zeta = w_m$ on $S_1$, one has
\begin{equation}\label{ate}
W(\zeta) > (1-\kappa_m)W(w_m).
\end{equation}
With such assumption in place we start by observing that there exists
$g_m=a_m\Re(\langle x',\nu_m\rangle+i |x_n|)^{3/2}\in H$ which
achieves the minimum
distance from $w_m$ to $H$:
\[
\|w_m-g_m\|_{W^{1,2}(B_1)}=\inf_{g\in H}\|w_m-g\|_{W^{1,2}(B_1)}.
\]
Indeed, this follows from the simple fact that the set $H$ is locally compact.
Combining this inequality with \eqref{ate0} we deduce that
$\|g_m-h\|_{W^{1,2}(B_1)}\le 2\theta_m$. As a consequence, we must
have that $\nu_m\rightarrow e_1$ and $a_m\rightarrow 1$. Hence,
\[
\left\|\frac{w_m}{a_m}-\Re(\langle x',\nu_m\rangle+i|x_n|)^{\frac{3}{2}}\right\|_{W^{1,2}(B_1)}\le \frac{\theta_m}{a_m} \to 0.
\]
If we rename $\frac{w_m}{a_m}\rightsquigarrow w_m$ and
$\frac{\theta_m}{a_m}\rightsquigarrow \theta_m$, and rotate
$\R^{n-1}$ to send $\nu_m$ to $e_1$, the renamed functions $w_m$ will
be homogeneous of degree $\frac{3}{2}$, nonnegative on $B_1'$, and
will satisfy
\begin{equation}\label{ate3}
\inf_{g\in H}\|w_m-g\|_{W^{1,2}(B_1)}=\|w_m-h\|_{W^{1,2}(B_1)}\le \theta_m.
\end{equation}
Moreover, \eqref{ate} will still hold for the renamed $w_m$,
because of the scaling property $W(tw)=t^2 W(w)$ and the invariance of
$W(w)$ under rotations in $\R^{n-1}$. 

We note explicitly that \eqref{ate} implies in particular that
$w_m\neq h$ for every $m\in \mathbb N$, as $W(h)=0$ (see Remark
\ref{R:weiss} above).
Thus we may also set
\begin{equation}\label{thetam}
\theta_m = \|w_m - h\|_{W^{1,2}(B_1)}>0
\end{equation}
for the rest of the proof.

We next want to rewrite \eqref{ate} in a slightly different way, using the
properties of function $h$. Given $\phi\in W^{1,2}(B_1)$, consider the
first variation of $W$ at $h$ in the direction of $\phi$
\begin{equation}\label{apples}
\delta W(h)(\phi) \defeq \int_{B_1}2\langle \nabla h,\nabla
\phi\rangle -\frac{3}{2}\int_{S_1}2h\phi,
\end{equation}
where the boundary integrals in \eqref{apples} and thereafter must be
interpreted in the sense of traces.
To compute $\delta W(h)(\phi)$ we write the first integral in the
right-hand side of \eqref{apples} as $\int_{B_1}2\langle \nabla
h,\nabla \phi\rangle =\int_{B_1^+}2\langle \nabla h,\nabla \phi\rangle
+\int_{B_1^-}2\langle \nabla h,\nabla \phi\rangle$. Now,
\begin{align*}
\int_{B_1^+}2\langle \nabla h,\nabla \phi\rangle &=-\int_{B_1^+}2(\Delta h)\phi +\int_{S_1^+}2\phi\langle \nabla h,\nu\rangle +\int_{B_1'}2\phi\langle \nabla h,\nu_+\rangle\\
&=\int_{S_1^+}2\phi\langle \nabla h,\nu\rangle+\int_{B_1'}2\phi\langle \nabla h,\nu_+\rangle, 
\end{align*}
where we have used the fact that $\Delta h = 0$ in $B_1^\pm$. Since
$h$ is homogeneous of degree $\frac 32$, by Euler's formula we have
$\langle \nabla h,\nu\rangle = \frac 32 h$ in $S_1^{\pm}$. Keeping in
mind that on $B_1'$ we have $\nu_\pm = \mp e_n$,  we find
\[
\int_{B_1^\pm}2\langle \nabla h,\nabla \phi\rangle = \frac 32 \int_{S_1^\pm}2\phi h + \int_{B_1'}2\phi\langle \nabla h,\nu_\pm\rangle = \frac 32 \int_{S_1^\pm}2\phi h \mp \int_{B_1'}2\phi \D_n^\pm h.
\]
Since $h$ is even in $x_n$, so that $\D_n^- h = - \D_n^+ h$ in $B_1'$,
we conclude
\begin{equation}\label{apples0}
\delta W(h)(\phi) = - 4\int_{B_1'}\phi \D_{n}^+ h.
\end{equation}
If now $\zeta\in W^{1,2}(B_1)$ is a function with $\zeta\ge 0$ on $B_1'$ and
such that $\zeta = w_m$ on $S_1$, by plugging in $\phi=\zeta-h$ into
\eqref{apples} and \eqref{apples0}, we obtain that
\begin{align*}
W(\zeta)&= W(\zeta)- W(h)- \delta W(h)(\zeta-h) -4\int_{B_1'} (\zeta-h) \D_{n}^+h 
\\
 &= \int_{B_1}|\nabla (\zeta-h)|^2-\frac{3}{2}\int_{S_1}(\zeta-h)^2-4\int_{B_1'} \zeta \D_n^+ h,
\end{align*}
where we have used that $W(h)=0$ and $h\D_n^+h=0$ on $B_1'$. By using a
similar identity for $W(w_m)$, we can rewrite \eqref{ate} as
\begin{multline}\label{theone}
(1-\kappa_m)\left[\int_{B_1}|\nabla (w_m-h)|^2-\frac{3}{2}\int_{S_1}(w_m-h)^2-4\int_{B_1'}w_m \D_{n}^+ h\right]
\\
<\int_{B_1}|\nabla (\zeta-h)|^2-\frac{3}{2}\int_{S_1}(\zeta-h)^2-4\int_{B_1'}\zeta \D_{n}^+h.
\end{multline}
Inequality \eqref{theone} will play a key role in the completion of
the proof and will be used repeatedly.

Let us introduce the normalized functions 
$$
\hat{w}_m=\frac{w_m-h}{\theta_m}.
$$ 
By \eqref{thetam} we have 
$$
\|\hat{w}_m\|_{W^{1,2}(B_1)}=1\quad\text{for every }m\in \mathbb N.
$$
By the weak compactness of the unit sphere in $W^{1,2}(B_1)$ we may
assume that 
$$
\hat{w}_m \to \hat{w}\quad\text{weakly in }W^{1,2}(B_1).
$$
Besides,
by the compactness of the Sobolev embedding and traces operator from
$W^{1,2}(B_1)$ into $L^2(B_1)$, $L^2(B_1')$, $L^2(S_1)$ (see e.g.\
Theorem 6.3 in \cite{Ne}), we may assume
$$
\hat{w}_m\to\hat{w}\quad\text{strongly in $L^2(B_1)$, $L^2(B_1')$, and $L^2(S_1)$}. 
$$
We then make the following

\begin{clm}\mbox{}
\begin{itemize}
\item[(i)] $\hat{w}\equiv 0$;
\item[(ii)] $\hat{w}_m\rightarrow 0$ strongly in $W^{1,2}(B_1)$.
\end{itemize}
\end{clm}

Note that (ii) will give us a contradiction since, by construction
$\|\hat{w}_m\|_{W^{1,2}(B_1)}=1$. Hence, the theorem will follow once
we prove the claim.

In what follows we will denote $\Lambda=\Lambda(h)$, the coincidence
set of $h$. 

\medskip
\noindent 
\emph{Step 1.} We start by showing that there is a constant $C>0$ such that 
\begin{equation}\label{sh1}
\Big\|\frac{w_m}{\theta_m^2} \D_n^+ h\Big\|_{L^1(B_1')} \le
C,\quad\text{for every}\ m\in \mathbb N.
\end{equation}
To this end, we pick a function $\eta\in W^{1,\infty}_0(B_1)$
such that $0<\eta\le 1$, and define $\zeta=(1-\eta)w_m+\eta h$. Then,
$\zeta=w_m$ on $S_1$ and $\zeta\ge 0$ on $B_1'$. Furthermore, $\zeta-h =
(1-\eta)(w_m-h)$. We can thus apply \eqref{theone} to such a $\zeta$,
obtaining 
\begin{multline*}
(1-\kappa_m)\left(\int_{B_1}|\nabla (w_m-h)|^2-\frac{3}{2}\int_{S_1}(w_m-h)^2-4\int_{B_1'}w_m \D_{n}^+ h\right)\\
< \int_{B_1}|\nabla ((1-\eta)(w_m-h))|^2-\frac{3}{2}\int_{S_1}(1-\eta)^2(h-w_m)^2-4\int_{B_1'}((1-\eta)w_m+\eta h) \D_{n}^+h\\
= \int_{B_1}\left[(1-\eta)^2|\nabla (w_m-h)|^2+|\nabla \eta|^2(w_m-h)^2-2(1-\eta)(w_m-h)\langle\nabla \eta,\nabla (w_m-h)\rangle\right]\\
-\frac{3}{2}\int_{S_1}(1-\eta)^2(h-w_m)^2-4\int_{B_1'}(1-\eta)w_m \D_{n}^+h.
\end{multline*}
Dividing by $\theta_m^2$, rearranging terms and using the fact that $\|\hat{w}_m\|_{W^{1,2}(B_1)} = 1$ and that $\D_n^+ h\le 0$ on $\Lambda$, we obtain
\begin{multline*}
4 \int_{B_1'} (\eta - \kappa_m) \frac{w_m}{\theta_m^2} |\D_{n}^+h|\le -(1-\kappa_m)\left(\int_{B_1}|\nabla \hat{w}_m|^2-\frac{3}{2}\int_{S_1}\hat{w}_m^2\right)\\
\qquad\qquad\mbox{}+\int_{B_1}\left[(1-\eta)^2|\nabla \hat{w}_m|^2+|\nabla
  \eta|^2\hat{w}_m^2-2(1-\eta)\hat{w}_m\langle\nabla \eta,\nabla \hat{w}_m\rangle\right]
  -\frac{3}{2}\int_{S_1}(1-\eta)^2\hat{w}_m^2\le C,
\end{multline*}
where $C>0$ is independent of $m\in \mathbb N$. At this point we choose $\eta(x)=\tilde{\eta}(|x|)$, and let  
\[
0 < \ve \defeq \int_0^1 \tilde \eta(r) r^n dr.
\]
Since $\kappa_m \to 0$ as $m\to \infty$, possibly passing to a
subsequence we can assume that $\kappa_m \le \frac \ve2(n+1)$ for
every $m\in \mathbb N$. With such choice we have
\[
\int_0^1(\tilde{\eta}(r)-\kappa_m)r^n dr \ge  \frac{\ve}{2},\quad m\in \mathbb N.
\]
Using the fact that $w_m$ and $h$ are homogeneous of degree $\frac
32$, we thus obtain
\begin{align*}
C\ge 4\int_{B_1'}(\eta-\kappa_m)\frac{w_m}{\theta_m^2} |\D_{n}^+ h| &= 4\left(\int_0^1(\tilde{\eta}(r)-\kappa_m)r^n dr\right)\int_{S_1'}\frac{w_m}{\theta_m^2}|\D_n^+ h| \ge 2\ve \int_{S_1'}\frac{w_m}{\theta_m^2}|\D_n^+ h|,
\end{align*}
which, again by the homogeneity of $w_m$ and $h$, and the fact that
$w_m\ge 0$ on $B_1'$, proves \eqref{sh1}.

\medskip\noindent
\emph{Step 2.} We next show that
\begin{equation}\label{deltaw}
\Delta \hat w = 0,\quad \text{in}\  B_1\setminus \Lambda.
\end{equation}
To establish \eqref{deltaw} it will suffice to show that for any ball $B$, such that its concentric double $2B\Subset B_1\setminus \Lambda$, and for any function $\phi\in W^{1,2}(B)$ such that $\phi - \hat{w} \in W^{1,2}_0(B)$, one has
\[
\int_B |\nabla \hat{w}|^2 \le \int_B |\nabla \phi|^2.
\]
To begin, we fix a function $\phi \in L^{\infty}(B_1)\cap W^{1,2}(B)$, and we consider 
\[
\zeta=\eta(h+\theta_m\phi)+(1-\eta)w_m, 
\]
where $\eta\in C^{\infty}_0(B_1\setminus \Lambda)$, $0\le\eta\le
1$. Notice that on $S_1$, $\zeta=w_m$, and, since $\phi\in
L^{\infty}(B_1)$ and $\eta\in C^{\infty}_0(B_1\setminus \Lambda)$, for
$m$ large enough we have $\zeta\ge 0$ on $B_1'$. For such sufficiently
large $m$'s, we can thus use the function $\zeta$ in \eqref{theone},
obtaining
\begin{multline*}
(1-\kappa_m)\left(\int_{B_1}|\nabla (w_m-h)|^2-\frac{3}{2}\int_{S_1}(w_m-h)^2 - 4\int_{B_1'}w_m \D_{n}^+ h\right) \\
<\int_{B_1}|\nabla ((1-\eta)(w_m-h)+\eta\theta_m\phi)|^2-\frac{3}{2}\int_{S_1}((1-\eta)(w_m-h)+\eta\theta_m\phi)^2\\
\mbox{}-4\int_{B_1'}(\eta(h+\theta_m\phi)+(1-\eta)w_m) \D_{n}^+h.
\end{multline*}
Dividing by $\theta_m^2$ and recalling that $h \D_{n}^+ h=0$ in
$B_1'$, we obtain
\begin{multline*}
(1-\kappa_m)\left(\int_{B_1}|\nabla \hat{w}_m|^2-\frac{3}{2}\int_{S_1}\hat{w}_m^2-4\int_{B_1'}\frac{w_m}{\theta_m^2} \D_{n}^+ h\right)\\
\shoveleft <\int_{B_1}\left[|\nabla (\eta\phi)|^2+|\nabla ((1-\eta)\hat{w}_m)|^2+2\langle\nabla(\eta\phi),\nabla((1-\eta)\hat{w}_m)\rangle\right]\\
\mbox{}-4\int_{B_1'}\left[\frac{\eta\phi}{\theta_m} \D_{n}^+ h+(1-\eta)\frac{w_m}{\theta_m^2} \D_{n}^+h\right] -\frac{3}{2}\int_{S_1}((1-\eta)\hat{w}_m+\eta\phi)^2\\
\shoveleft =\int_{B_1}\left[|\nabla (\eta\phi)|^2+|\nabla
  ((1-\eta)\hat{w}_m)|^2+2\langle\nabla(\eta\phi),\nabla((1-\eta)\hat{w}_m)\rangle\right]
-4\int_{B_1'}\frac{w_m}{\theta_m^2} \D_{n}^+ h -\frac{3}{2}\int_{S_1}\hat{w}_m^2,
\end{multline*}
since $\eta\in C^{\infty}_0(B_1\setminus \Lambda)$. Hence
\begin{align*}
\int_{B_1}|\nabla \hat{w}_m|^2&< \kappa_m\int_{B_1}|\nabla \hat{w}_m|^2+\frac{3}{2}(1-\kappa_m)\int_{S_1}\hat{w}_m^2+4(1-\kappa_m)\int_{B_1'}\frac{w_m}{\theta_m^2} \D_{n}^+h\\
&\qquad\mbox{}+\int_{B_1}\left[|\nabla (\eta\phi)|^2+|\nabla((1-\eta)\hat{w}_m)|^2+2\langle\nabla(\eta\phi),\nabla((1-\eta)\hat{w}_m)\rangle\right] \\
&\qquad\mbox{}-4\int_{B_1'}\frac{w_m}{\theta_m^2} \D_{n}^+h-\frac{3}{2}\int_{S_1}\hat{w}_m^2\\
&=\kappa_m\int_{B_1}|\nabla \hat{w}_m|^2-\frac{3}{2}\kappa_m\int_{S_1}\hat{w}_m^2-4\int_{B_1'}\kappa_m\frac{w_m}{\theta_m^2} \D_{n}^+h\\
&\qquad\mbox{}+\int_{B_1}\left[|\nabla (\eta\phi)|^2+|\nabla((1-\eta)\hat{w}_m)|^2+2\langle\nabla(\eta\phi),\nabla((1-\eta)\hat{w}_m)\rangle\right]\\
&\le C\kappa_m+\int_{B_1}\left[|\nabla (\eta\phi)|^2+|\nabla((1-\eta)\hat{w}_m)|^2+2\langle\nabla(\eta\phi),\nabla((1-\eta)\hat{w}_m)\rangle\right].
\end{align*}
Therefore

\begin{align*}
\int_{B_1}(1-(1-\eta)^2)|\nabla \hat{w}_m|^2&\le C\kappa_m+\int_{B_1}\Big[|\nabla (\eta\phi)|^2+\hat{w}_m^2|\nabla\eta|^2\\&\qquad\mbox{}-2(1-\eta)\hat{w}_m\langle\nabla\eta,\nabla \hat{w}_m\rangle+2\langle\nabla(\eta\phi),\nabla((1-\eta)\hat{w}_m)\rangle\Big].
\end{align*}
Passing to the limit $m\rightarrow\infty$ we obtain
\begin{equation}\label{mmm}\begin{aligned}
\int_{B_1}(1-(1-\eta)^2)|\nabla \hat{w}|^2&\le \int_{B_1}\Big[|\nabla
(\eta\phi)|^2+\hat{w}^2|\nabla\eta|^2\\
&\qquad\mbox{}-2(1-\eta)\hat{w}\langle\nabla\eta,\nabla \hat{w}\rangle+2\langle\nabla(\eta\phi),\nabla((1-\eta)\hat{w})\rangle\Big].
\end{aligned}
\end{equation}
Notice that 
\begin{align*}
\int_{B_1}|\nabla(\eta\phi+(1-\eta)\hat{w})|^2&=\int_{B_1}\left[|\nabla(\eta\phi)|^2+|\nabla((1-\eta)\hat{w})|^2+2\langle \nabla(\eta\phi),\nabla((1-\eta)\hat{w}\rangle\right]\\
&=\int_{B_1}\Big[|\nabla(\eta\phi)|^2+\hat{w}^2|\nabla\eta|^2+(1-\eta)^2|\nabla \hat{w}|^2-2\hat{w}(1-\eta)\langle \nabla \hat{w},\nabla \eta\rangle\\
&\qquad\mbox{}+2\langle \nabla(\eta\phi),\nabla((1-\eta)\hat{w})\rangle\Big],
\end{align*}
hence \eqref{mmm} gives us,
\[
\int_{B_1}|\nabla \hat{w}|^2\le \int_{B_1}|\nabla(\eta\phi+(1-\eta)\hat{w})|^2.
\]
By approximation we can drop the condition $\phi\in L^{\infty}(B_1)$
and by considering open balls $B\Subset B_1\setminus \Lambda$ we may
choose $\eta=1$ in $B$ and $\phi=\hat{w}$ outside $B$. This will give
\[
\int_{B_1}|\nabla \hat{w}|^2\le \int_B|\nabla \phi|^2+\int_{B_1\setminus B}|\nabla \hat{w}|^2,
\]
hence
\[
\int_B|\nabla \hat{w}|^2\le \int_B|\nabla\phi|^2,
\]
which proves the harmonicity of $\hat{w}$ in $B$.

\medskip
\noindent 
\emph{Step 3.} We next want to prove that
\begin{equation}\label{step3}
\hat w= 0\quad \mathcal{H}^{n-1}\text{-a.e.\ in}\  \Lambda.
\end{equation}
We note that for the function $h$ we have $\D_n^+ h(x',0)\not=0$ for
every $(x',0)\in \Lambda^\circ$ (interior of $\Lambda$ in
$\R^{n-1}$). Therefore, given $\omega \Subset \Lambda^\circ$, there
exists a constant $C_\omega>0$ such that $|\D_n^+ h(x',0)|\ge
C_\omega$ for every $(x',0)\in \omega$.
At points $(x',0)\in \Lambda^\circ$, we can thus write
\[
\hat{w}_m = \frac{w_m - h}{\theta_m} = \frac{w_m}{\theta_m^2} \D_n^+ h \frac{\theta_m}{\D_n^+ h}.
\]
This gives
\[
\int_\omega |\hat{w}_m| \le \frac{\theta_m}{C_\omega} \int_\omega \frac{w_m}{\theta_m^2} |\D_n^+ h|\le \frac{C\theta_m}{C_\omega},
\]
where in the last inequality we have used \eqref{sh1} in Step 1
above. Since $\theta_m \to 0$, we conclude that
$\|\hat{w}_m\|_{L^1(\omega)} \to 0$ as $m\to \infty$. By the
arbitrariness of $\omega\Subset \Lambda^\circ$ we infer that, in
particular, we must have
\begin{equation}\label{step32}
\hat{w}_m(x',0)\to 0,\quad\mathcal H^{n-1}\text{-a.e.}\ (x',0)\in \Lambda,
\end{equation}
which proves \eqref{step3}.

\medskip\noindent
\emph{Step 4 (Proof of (i)).} We next show that
\begin{equation}\label{weakly}
\hat{w}_m \to 0,\quad\text{weakly in}\ W^{1,2}(B_1),
\end{equation}
or, equivalently, $\hat{w}=0$. We begin by observing that, since the $\hat{w}_m$'s are homogeneous of degree 
$\frac{3}{2}$, 
their weak limit $\hat{w}$ is also homogeneous of degree
$\frac{3}{2}$. Combining this observation with Steps 2 and 3 above, we
then have the following properties for $\hat{w}$:
\begin{itemize}
\item[(i)] $\Delta \hat{w}=0$ in $B_1\setminus\Lambda$;
\item[(ii)] $\hat{w}=0$ $\mathcal{H}^{n-1}$-a.e.\ on $\Lambda$;
\item[(iii)] $\hat{w}$ is homogeneous of degree $\frac{3}{2}$.
\end{itemize}
We next obtain an explicit representation for $\hat{w}$. First, we
note that $\hat{w}$ is H\"older continuous up to the coincidence set
$\Lambda$ of $h$. Indeed, this can be seen
by making a bi-Lipschitz transformation $T:B_1\setminus \Lambda\to
B_1^+$ as in (b) on p. 501 of \cite{AC1}. The function $\tilde
w=\hat{w}\circ T^{-1}\in W^{1,2}(B_1^+)$ solves a uniformly elliptic
equation in divergence form
$$
\div (b(x)\nabla \tilde w)=0\quad\text{in }B_1^+
$$
with bounded measurable coefficients $b(x)=[b_{ij}(x)]$. We will also have
$$
\tilde w=0\quad\text{on }B_1'
$$
in the sense of traces. Then, by the boundary version of De
Giorgi-Nash-Moser regularity theorem (see, e.g., Theorem 8.29 in
\cite{GT})
we have that $\tilde w$ is $C^\gamma$ up to $B_1'$ for some
$\gamma>0$. Since $\hat{w}(x) = \tilde w(T(x))$, this implies that
$\hat{w}\in C^\gamma(B_1)$.

Once we know that $\hat{w}\in C^\gamma(B_1)$, together with (i)--(iii) above,
we can apply a theorem of De~Silva and Savin on an expansion of
harmonic functions in
slit domains, see Theorem~3.3 in \cite{DS} (and also Theorem~4.5 in the same paper) which implies that there are constants $a_1$, $a_2$, \ldots, $a_{n-1}$, $b$, and $c$ 
such that, for some $\alpha\in (0,1)$,
$$
|\hat{w}(x)-P_0(x)U_0(x)-c x_n|=O(|x|^{3/2+\alpha}),
$$
 where
\[
U_0(x)=\frac{1}{\sqrt{2}}\sqrt{x_{1}+\sqrt{x_{1}^2+x_{n}^2}}=\Re(x_{1}+i|x_n|)^{1/2},
\]
and
\[
P_0(x)=\sum_{k=1}^{n-1}a_k x_{k}+b\sqrt{x_{1}^2+x_n^2}.
\]
Since $\hat{w}$ is $3/2$-homogeneous, we must have $c=0$ and thus
$$
\hat{w}(x)=P_0(x)U_0(x).
$$
Now, a direct computation shows that such $\hat{w}$ will be harmonic in
$B_1\setminus\Lambda$ only if 
$$
a_{1}+2b=0,
$$
which implies the representation
\begin{align*}
\hat{w}(x)&=\frac{a_1}{2}\Re(x_{1}+i |x_n|)^{3/2}+\sum_{j=2}^{n-1}
a_{j}x_j\Re(x_{1}+i |x_n|)^{1/2},\\
&=\frac{a_1}{2} h(x)+\sum_{j=2}^{n-1} a_j x_j U_0(x).
\end{align*}
We next show that all constants $a_j=0$, $j=1,\ldots,n-1$. To simplify
the notation we will write $\|\cdot\|=\|\cdot\|_{W^{1,2}(B_1)}$. We
then use the
fact that 
\[
\|w_m-g\|^2\ge \|w_m-h\|^2\quad\text{for all }\ g\in H.
\]
Recalling that $\hat{w}_m=\frac{w_m-h}{\theta_m}$, we can write this as
\[
\|\theta_m\hat{w}_m+h-g\|^2\ge \|\theta_m \hat{w}_m\|^2,
\]
or
\[
2\theta_m\langle \hat{w}_m,h-g\rangle +\|h-g\|^2\ge 0.
\]
Therefore,
\begin{equation}\label{inequality}
\langle \hat{w}_m,g-h\rangle \le \frac{\|h-g\|^2}{2\theta_m}.
\end{equation}
Applying this to $g=(1+\theta_m^2)h$, we obtain
\[
\langle \hat{w}_m, h\rangle \le \frac{\theta_m}{2}\|h\|^2.
\]
Letting $m\rightarrow \infty$ we arrive at
\[
\langle \hat{w},h\rangle = \frac{a_1}{2}\|h\|^2\le 0.
\]
This implies that $a_{1}\leq 0$. Using the same argument for
$g=(1-\theta_m^2)h$ allows us to conclude
that also $- a_{1}\le 0$, and therefore $a_{1}=0$. Further,
rewriting \eqref{inequality} as
$$
\left\langle \hat{w}_m, \frac{g-h}{\theta_m^2} \right\rangle\leq
\frac{\theta_m}2\left\| \frac{g-h}{\theta_m^2}\right\|^2,
$$
and taking for $j=2,\ldots, n-1$
$$
g=\Re(x_{1}\cos(\theta_m^2)+\sin(\theta_m^2)x_j+i|x_n|)^{3/2},
$$
in such inequality, by letting $m\to \infty$ we obtain that
$$
\frac32\langle \hat{w}, x_j U_0\rangle=\frac{3}{2}a_j \|x_j U_0\|^2\leq 0.
$$
(We note here that $\langle x_i U_0, x_j U_0\rangle=0$ for $i,
j=2,\ldots, n-1$, $i\neq j$ and that 

$$
\frac{\Re(x_{1}\cos(\theta)+\sin(\theta)x_j+i|x_n|)^{3/2}-\Re(x_{1}+i|x_n|)^{3/2}}{\theta}\to
\frac32 x_jU_0(x) 
$$
as $\theta\to 0$, strongly in $W^{1,2}(B_1)$.)
Hence $a_j\leq 0$. Replacing $x_j$ with $-x_j$ in the above
argument, we also obtain $-a_j\leq 0$. Thus, $a_j=0$ for
all $j=1,\ldots, n-1$, which implies $\hat{w}=0$ and completes the
proof of \eqref{weakly}.

\medskip\noindent
\emph{Step 5 (Proof of (ii)):} Finally, we claim that, on a subsequence, 
\begin{equation}\label{step5}
\hat{w}_m\rightarrow 0\quad\text{strongly in} \ W^{1,2}(B_1).
\end{equation}
Since we already have the strong convergence $\hat{w}_m\to \hat{w}=0$
in $L^2(B_1)$, we are left with proving 
\begin{equation}\label{step52}
\nabla \hat{w}_m\rightarrow 0\quad\text{strongly in}\  L^2(B_1).
\end{equation}
To this end, we pick $\eta\in C^{0,1}_0(B_1)$, $0\le \eta\le 1$,  and
consider $\zeta=(1-\eta)w_m+\eta h$. Clearly, $\zeta=w_m$ on $S_1$,
$\zeta\ge 0$ on $B_1'$, and $\zeta- h = (1-\eta)(w_m-h)$. Applying
\eqref{theone} with this choice of $\zeta$ we obtain
\begin{multline*}
 (1-\kappa_m)\left[\int_{B_1}|\nabla (w_m-h)|^2-\frac{3}{2}\int_{S_1}(w_m-h)^2-4\int_{B_1'}w_m \D_{n}^+ h\right]\\
 < \int_{B_1}|(1-\eta)\nabla (w_m-h)-\nabla \eta (w_m-h)|^2
 \\-\frac{3}{2}\int_{S_1}(1-\eta)^2(w_m-h)^2
 \mbox{}- 4 \int_{B_1'}(1-\eta)w_m \D_{n}^+h.
\end{multline*}
Dividing by $\theta_m^2$, and recalling that
$\hat{w}_m=\frac{w_m-h}{\theta_m}$, we obtain
\begin{multline*}
(1-\kappa_m)\left(\int_{B_1}|\nabla \hat{w}_m|^2-\frac{3}{2}\int_{S_1}\hat{w}_m^2-4\int_{B_1'}\frac{w_m}{\theta_m^2} \D_{n}^+h\right)
\\
 < \int_{B_1}\left[(1-\eta)^2|\nabla \hat{w}_m|^2+\hat{w}_m^2|\nabla \eta|^2-2(1-\eta)\hat{w}_m\langle \nabla \hat{w}_m,\nabla \eta\rangle\right]\\ -\frac{3}{2}\int_{S_1}(1-\eta)^2\hat{w}_m^2- 4 \int_{B_1'}(1-\eta)\frac{w_m}{\theta_m^2} \D_{n}^+h.
\end{multline*}
This gives
\begin{multline*}
\int_{B_1}|\nabla \hat{w}_m|^2-4\int_{B_1'}\frac{w_m}{\theta_m^2} \D_{n}^+h\\
\le \int_{B_1}\left[(1-\eta)^2|\nabla \hat{w}_m|^2+|\nabla \eta|^2\hat{w}_m^2-2(1-\eta)\hat{w}_m\langle \nabla \hat{w}_m,\nabla \eta\rangle\right]\\
 - \frac{3}{2}\int_{S_1}(1-\eta)^2\hat{w}_m^2-4\int_{B_1'}(1-\eta)\frac{w_m}{\theta_m^2} \D_{n}^+h+(1-\kappa_m)\frac{3}{2}\int_{S_1}\hat{w}_m^2\\
+ \kappa_m\left(\int_{B_1}|\nabla \hat{w}_m|^2-4\int_{B_1'}\frac{w_m}{\theta_m^2}\D_{n}^+h\right)
\\
=\int_{B_1}\left[(1-\eta)^2|\nabla \hat{w}_m|^2+|\nabla
\eta|^2\hat{w}_m^2-2(1-\eta)\hat{w}_m\langle \nabla \hat{w}_m,\nabla \eta\rangle\right]\\
 +\kappa_m\left(\int_{B_1}|\nabla
  \hat{w}_m|^2-4\int_{B_1'}\frac{w_m}{\theta_m^2} \D_{n}^+h-\frac{3}{2}\int_{S_1}\hat{w}_m^2\right)
    \\
 + \frac{3}{2}\left(1-(1-\eta)^2\right)\int_{S_1}\hat{w}_m^2- 4 \int_{B_1'}(1-\eta)\frac{w_m}{\theta_m^2} \D_{n}^+h.
\end{multline*}
If in this inequality we use the fact that $\|\nabla
\hat{w}_m\|_{L^2(B_1)}\le \|\hat{w}_m\|_{W^{1,2}(B_1)} = 1$, and that
$\frac{w_m}{\theta_m^2}\D_{n}^+h$ is uniformly bounded in $L^1(B_1')$,
a fact which we have proved in \eqref{sh1} of Step 1, we obtain
\begin{multline}\label{step53}
\int_{B_1}|\nabla \hat{w}_m|^2-4\int_{B_1'}\frac{w_m}{\theta_m^2} \D_{n}^+h
\\
\le \int_{B_1}(1-\eta)^2|\nabla \hat{w}_m|^2+|\nabla \eta|^2\hat{w}_m^2-2(1-\eta)\hat{w}_m\langle \nabla \hat{w}_m,\nabla \eta\rangle
\\
 + C\kappa_m  +\frac{3}{2}\int_{S_1}\hat{w}_m^2 -4\int_{B_1'}(1-\eta)\frac{w_m}{\theta_m^2} \D_{n}^+h.
\end{multline}
We now make the choice in \eqref{step53} of
\[
\eta(x)=\begin{cases} 
1,& \text{if } |x|\le \frac{1}{2},
\\
2(1-|x|),& \text{if }\frac{1}{2}< |x| < 1,
\\
0,& \text{if } |x|\ge 1,
\end{cases}
\]
we obtain
\begin{align*}
\int_{B_{\frac{1}{2}}}|\nabla \hat{w}_m|^2& \le \int_{B_1}\left[|\nabla \eta|^2\hat{w}_m^2-2(1-\eta)\hat{w}_m\langle\nabla \hat{w}_m,\nabla \eta\rangle\right] +\frac{3}{2}\int_{S_1}\hat{w}_m^2
+ C\kappa_m +4\int_{B_1'}\eta\frac{w_m}{\theta_m^2} \D_{n}^+h\\
&\le \int_{B_1}\left[|\nabla \eta|^2\hat{w}_m^2-2(1-\eta)\hat{w}_m\langle\nabla \hat{w}_m,\nabla \eta\rangle\right] +\frac{3}{2}\int_{S_1}\hat{w}_m^2
+ C\kappa_m,
\end{align*}
since $\eta, w_m\ge 0$ and $\D_{n}^+h \le 0$. We thus conclude that
\begin{equation}\label{eta}
\int_{B_{\frac{1}{2}}}|\nabla \hat{w}_m|^2 \le \int_{B_1}\left[|\nabla \eta|^2\hat{w}_m^2-2(1-\eta)\hat{w}_m\langle \nabla \hat{w}_m,\nabla \eta\rangle\right] +\frac{3}{2}\int_{S_1}\hat{w}_m^2
+ C\kappa_m.
\end{equation}
We now observe that, since $\hat{w}_m$ is homogeneous of degree $3/2$,
and thus $\nabla \hat{w}_m$ is homogeneous of degree $1/2$, we have
\[
\int_{B_{1}}|\nabla \hat{w}_m|^2= 2^{n+1} \int_{B_{\frac{1}{2}}}|\nabla \hat{w}_m|^2.
\]
Using this identity in \eqref{eta} we conclude that
\[
\int_{B_1}|\nabla \hat{w}_m|^2 \le 2^{n+1}\left(\int_{B_1}\left[|\nabla \eta|^2\hat{w}_m^2-2(1-\eta)\hat{w}_m\langle \nabla \hat{w}_m,\nabla \eta\rangle\right] +\frac{3}{2}\int_{S_1}\hat{w}_m^2 +C\kappa_m\right).
\]
To complete the proof of \eqref{step5}, and consequently of Theorem
\ref{T:epi}, all we need to do at this point is to observe that, on a
subsequence, the right-hand side of the latter inequality converges to
$0$ as $m\to \infty$. This follows from the facts that $\kappa_m\to 0$,
$\|\hat{w}_m\|_{L^2(B_1)}\to 0$, $\|\hat{w}_m\|_{L^2(S_1)}\to 0$, and $\|\nabla \hat{w}_m\|_{L^2(B_1)} \le
1$. 

This completes the proof of the claim and that of the theorem.
\end{proof}

\section{$C^{1,\beta}$ regularity of the regular part of the free boundary}\label{S:final}

In this final section we combine Theorems \ref{T:weiss} and \ref{T:epi} to establish the $C^{1,\beta}$ regularity of the regular part of the free boundary.
We will consider two types of scalings: the Almgren one, defined in \eqref{Almgrentype}, and the homogeneous scalings defined in \eqref{s}, which are suited for the study of regular free boundary points, i.e.,
\begin{equation}\label{hom2}
v_r(x)=\frac{v(rx)}{r^{3/2}}.
\end{equation}
Throughout this section we continue to use the notation
$h(x)=\Re(x_1+i|x_n|)^{\frac{3}{2}}$ adopted in Section
\ref{S:epi}. The symbol $\theta>0$ will be used to exclusively denote
the constant in the epiperimetric inequality of Theorem~\ref{T:epi}
above. 

 In Lemma~\ref{L:Wbounded} above we showed that our Weiss type
 functional $W_L(v,r)$ is bounded,  when $v$ is the solution to the problem \eqref{c11i}--\eqref{c44i}. In the course of the proof of the next
 lemma we establish the much more precise statement that $W_L(v,r) \le
 C r^{\gamma}$, for appropriate constants $C, \gamma>0$. This gain is possible
 because of the assumption, in Lemma~\ref{L:reg} below, that the
 scalings $v_r$ have the epiperimetric property, i.e., the conclusion
 of the epiperimetric inequality holds for their extensions as
 $3/2$-homogeneous functions in $B_1$. 

\begin{lemma}\label{L:reg} Let $v$ be the solution of the thin
  obstacle problem \eqref{c11i}--\eqref{c44i}, and suppose that $0\in
  \Gamma_{3/2}(v)$. Assume the existence of radii $0\le s_0<r_0<1$
  such that for every $s_0\le r\le r_0$, if we extend $v_r\big|_{S_1}$
  as a $3/2$-homogeneous function in $B_1$, call it $w_r$, then there
  exists a function $\zeta_r \in W^{1,2}(B_1)$ such that $\zeta_r \ge
  0$ in $B_1'$, $\zeta_r =v_r$ on $S_1$ and
\[
W(\zeta_r)\le (1-\kappa)W(w_r),
\]
where $\kappa$ is the constant in the epiperimetric inequality. 
Then, there exist universal constants $C, \gamma>0$ such that for every $s_0\le s\le t\le r_0$ one has
\begin{equation}\label{maincon}
\int_{S_1}|v_t-v_s| \le Ct^{\gamma}.
\end{equation}
\end{lemma}

\begin{proof} As before,  we let $L=\text{div}(A\nabla \cdot)$. The
  main idea of the proof of \eqref{maincon} is to relate
  $\int_{S_1}|v_t-v_s|$ with the Weiss type functional $W_L(v,r)$
  defined in \eqref{W} above, and then control the latter in the
  following way:
\begin{equation}\label{WW}
W_L(v,t) \le C t^{\gamma},\quad 0<t<r_0. 
\end{equation}

More specifically, combining equations \eqref{bb} and \eqref{ma}
proved below, we obtain the following
\begin{equation}\label{relat}
\int_{S_1}|v_t-v_s|\le C(n)\int_s^tr^{-1/2}\left( \frac{d}{dr}W_L(v,r)+Cr^{-1/2}\right)^{1/2}dr.
\end{equation}
After using H\"older's inequality in the right-hand side of
\eqref{relat} we obtain
\begin{align*}
\int_{S_1}|v_t-v_s| & \le C\left(\int_s^t r^{-1}dr\right)^{1/2}\left(\int_s^t \frac{d}{dr}W_L(v,r)+Cr^{-1/2} dr\right)^{1/2}\\
& \le C\left(\log \frac{t}{s}\right)^{1/2}\left(W_L(v,t)-W_L(v,s)+C(t^{1/2}-s^{1/2})\right)^{1/2}\\
& \le C\left(\log \frac{t}{s}\right)^{1/2}\left(Ct^{\gamma} +Cs^{1/2}+C(t^{1/2}-s^{1/2})\right)^{1/2}\\
&\le C\left(\log \frac{t}{s}\right)^{1/2}t^{\frac{\gamma}{2}},
\end{align*}
where we have used \eqref{WW}, and we have estimated $-W_L(v,s) \le  Cs^{1/2}$ using Corollary \ref{C:weiss}.
With this estimate in hands we now use a dyadic argument. Assume that $s\in [2^{-\ell},2^{-\ell+1})$, $t\in [2^{-h},2^{-h+1})$ with $h\le \ell$ and apply the estimate above iteratively. We obtain 
\begin{align*}
\int_{S_1}|v_s-v_t| & \le \int_{S_1}|v_s-v_{2^{-\ell+1}}|+ \cdots + \int_{S_1}|v_{2^{-h}}-v_t|\\
&\le C \log^{1/2} \left(\frac{2^{-\ell+1}}{s} \right)\left(2^{-\ell+1}\right)^{\frac{\gamma}{2}} + \cdots + C\log^{1/2}\left(\frac{t}{2^{-h}}\right)t^{\frac{\gamma}{2}}\\
&\le C(\log 2)^{1/2}\left(2^{\frac{\gamma}{2}}\right)^{-\ell+1}+\cdots + C(\log 2)^{1/2}t^{\frac{\gamma}{2}}\\
&\le C(\log 2)^{1/2}\sum_{j=h}^{\ell-1}\left(2^{\frac{\gamma}{2}}\right)^{-j} +C(\log 2)^{1/2}t^{\frac{\gamma}{2}} \\
&\le C(\log 2)^{1/2}\left(2^{\frac{\gamma}{2}}\right)^{-h}+C(\log 2)^{1/2}t^{\frac{\gamma}{2}}\le Ct^{\frac{\gamma}{2}},
\end{align*}
which yields the sought for conclusion \eqref{maincon}.

In order to complete the proof of the lemma we are thus left with
proving \eqref{WW} and \eqref{relat}. Our first step will be to prove
\eqref{WW} since the computations leading to such estimate also give
\eqref{relat}, as we will show below. We will establish \eqref{WW} by
proving (see \eqref{WWW} below) the following estimate 
\[
\frac{d}{dr}W_L(v,r)\ge \frac{n+1}{r}\frac{\kappa}{1-\kappa}W_L(v,r)-Cr^{-1/2},
\]
where $\kappa$ is the constant in the epiperimetric inequality. With
this objective in mind, we recall that
$$
W_L(v,r)=\frac{I_L(v,r)}{r^{n+1}}-\frac{3}{2}\frac{H_L(v,r)}{r^{n+2}},
$$
see \eqref{WW'}. To simplify the notation we write $I=I_L$ and
$H=H_L$. We start by observing that combining Lemma~\ref{L:H'i} with
the observation that $L|x|= \div (A(x)\nabla
r)=\frac{n-1}{|x|}(1+O(|x|))$ (see Lemma 4.1 in
\cite{GS}), we obtain the following estimate for $H'(r)$:
\begin{equation}\label{H'}
H'(r)-\left(\frac{n-1}{r}+ O(1)\right)H(r) = 2 I(r).
\end{equation}
In the computations that follow we will estimate
$\frac{d}{dr}W_L(v,r)$ using formula \eqref{H'}, estimates
\eqref{vbounds},  \eqref{mf4}, as well as the identity
$I(r)=D(r)+\int_{B_r}vf$, which gives $I'(r) = \int_{S_r} \langle
A\nabla v,\nabla v\rangle + \int_{S_r} vf$.
We thus have
\begin{align*}
\frac{d}{dr}W_L(v,r)& \stackrel{\eqref{WW'}}{=}\frac{I'(r)}{r^{n+1}}-\frac{n+1}{r^{n+2}}I(r)-\frac{3}{2r^{n+2}}H'(r)+\frac{3(n+2)}{2r^{n+3}}H(r)\\
&\stackrel{\eqref{H'}}{\ge}\frac{1}{r^{n+1}}\int_{S_r}\langle A\nabla v,\nabla v\rangle -\frac{n+1}{r^{n+2}}D(r)+\frac{3(n+2)}{2r^{n+3}}H(r)\\
&\qquad\mbox{}-\frac{3}{2r^{n+2}}\left(\frac{n-1}{r}H(r)+2\int_{S_r}v\langle
  A\nu,\nabla v\rangle
  +CH(r)\right)\\
&\qquad\mbox{}+\frac{1}{r^{n+1}}\int_{S_r}
vf-\frac{n+1}{r^{n+2}}\int_{B_r} vf\\
&\stackrel{\eqref{vbounds}}{\ge} \frac{1}{r^{n+1}}\int_{S_r}\langle A\nabla v,\nabla v\rangle-\frac{n+1}{r^{n+2}}D(r)+\frac{9}{2r^{n+3}}H(r) -\frac{3}{r^{n+2}}\int_{S_r}v\langle A\nu,\nabla v\rangle -Cr^{-1/2}\\
&\stackrel{\phantom{
\eqref{vbounds}}}{=}\frac{1}{r^{n+1}}\int_{S_r}\langle A\nabla v,\nabla v\rangle -\frac{n+1}{r}W_L(v,r)-\frac{3(n-2)}{2r^{n+3}}H(r)-\frac{3}{r^{n+2}}\int_{S_r}v\langle A\nu,\nabla v\rangle -Cr^{-1/2},
\end{align*}
where using \eqref{vbounds} we have estimated $CH(r)\le C r^{n+2}$, $|\int_{S_r} vf|\le C r^{n+\frac 12}$.
Now
\begin{align*}
H(r)&=\int_{S_r}\mu v^2=\int_{S_r}v^2+\int_{S_r}(\mu-1)v^2\\
&\le \int_{S_r}v^2+r\int_{S_r}v^2\stackrel{\eqref{vbounds}}{\le} \int_{S_r}v^2+Crr^{n-1+3}=\int_{S_r}v^2+Cr^{n+3}.
\end{align*}
Similarly,
\begin{align*}
\int_{S_r}\langle A\nabla v,\nabla v\rangle &= \int_{S_r}|\nabla
v|^2+\int_{S_r}\langle (A(x)-A(0))\nabla v,\nabla v\rangle\\
& \le \int_{S_r}|\nabla v|^2+r\int_{S_r}|\nabla v|^2\stackrel{\eqref{vbounds}}{\le} \int_{S_r}|\nabla v|^2+Cr^{n+1}.
\end{align*}
Finally,
\begin{align*}
\int_{S_r}v\langle A\nu,\nabla v\rangle &= \int_{S_r}v\langle
\nu,\nabla v\rangle+\int_{S_r}v\langle (A(x)-A(0))\nu,\nabla
v\rangle\\
&\le \int_{S_r}v\langle \nu,\nabla v\rangle+ Cr\int_{S_r}|v| |\nabla v|\\
&\le \int_{S_r}v\langle \nu,\nabla v\rangle
+Cr\left(\int_{S_r}v^2\int_{S_r}|\nabla v|^2\right)^{\frac{1}{2}}\\
&\stackrel{\eqref{vbounds}}{\le}  \int_{S_r}v\langle \nu,\nabla v\rangle+Crr^{\frac{n-1+3}{2}}r^{\frac{n-1+1}{2}}\\
&=\int_{S_r}v\langle \nu,\nabla v\rangle+Cr^{n+2}.
\end{align*}
This implies
\begin{align*}
\frac{d}{dr}W_L(v,r)&\ge \frac{1}{r^{n+1}}\int_{S_r}|\nabla v|^2-\frac{n+1}{r}W_L(v,r)-\frac{3(n-2)}{2r^{n+3}}\int_{S_r}v^2-\frac{3}{r^{n+2}}\int_{S_r}v\langle \nu,\nabla v\rangle -Cr^{-1/2}\\
&=-\frac{n+1}{r}W_L(v,r)+\frac{1}{r}\int_{S_1}|\nabla v_r|^2-\frac{3(n-2)}{2r}\int_{S_1}v_r^2-\frac{3}{r}\int_{S_1}v_r\langle \nu,\nabla v_r\rangle -Cr^{-1/2}\\
&=-\frac{n+1}{r}W_L(v,r) +\frac{1}{r}\int_{S_1}\left(
  (\langle \nu,\nabla v_r\rangle-\frac{3}{2} v_r)^2+|\D_{\tau}
  v_r|^2-\frac{3}{2}\left(n-\frac{1}{2}\right)v_r^2\right) -C r^{-1/2},
\end{align*}
where $\D_{\tau}v_r$ is the tangential derivative of $v_r$ along
$S_1$. Let $w_r$ be the $\frac{3}{2}$-homogeneous extension of
$v_r\big|_{S_1}$, then

\[
\int_{S_1}\left(|\D_{\tau} v_r|^2-\frac{3}{2}\left(n-\frac{1}{2}\right)v_r^2\right) =\int_{S_1}\left(|\D_{\tau} w_r|^2-\frac{3}{2}\left(n-\frac{1}{2}\right)w_r^2\right).
\]
Recalling that $w_r$ is homogeneous of degree $3/2$, we have on $S_1$ that $\langle \nabla w_r,\nu\rangle = \frac 32 w_r$. This gives
\begin{multline*}
\int_{S_1}\left(|\D_{\tau} w_r|^2-\frac{3}{2}\left(n-\frac{1}{2}\right)w_r^2\right)=\int_{S_1}\left(|\nabla w_r|^2-\langle\nabla w_r,\nu\rangle^2-\frac{3}{2}\left(n-\frac{1}{2}\right)w_r^2\right)\\
=\int_{S_1}\left(|\nabla w_r|^2-\frac{9}{4}w_r^2-\frac{3}{2}\left(n-\frac{1}{2}\right)w_r^2\right)=\int_{S_1}\left(|\nabla w_r|^2-\frac{3}{2}\left(n+1\right)w_r^2\right).
\end{multline*}
Now, since again by the homogeneity of $w_r$,
\[
\int_{S_1}|\nabla w_r|^2=(n+1)\int_{B_1}|\nabla w_r|^2,
\]
we conclude that
\[
\int_{S_1}\left(|\D_{\tau} w_r|^2-\frac{3}{2}\left(n-\frac{1}{2}\right)w_r^2\right)=(n+1)W(w_r,1),
\]
where we recall that, by definition \eqref{W} above,
\[
W(w,s)=W_\Delta(w,s)=\frac{1}{s^{n+1}}\int_{B_s}|\nabla w|^2-\frac{3}{2s^{n+2}}\int_{S_s}w^2.
\]
Hence
\begin{equation}\label{bdW'}
\frac{d}{dr}W_L(v,r)\ge \frac{n+1}{r}\left(W(w_r,1)-W_L(v,r)\right)+\frac{1}{r}\int_{S_1}(\langle \nu,\nabla v_r\rangle - \frac{3}{2} v_r)^2 -Cr^{-1/2}.
\end{equation}
By the hypothesis, for every $s_0\le r\le r_0$ there exists a function
$\zeta_r\in W^{1,2}(B_1)$ such that $\zeta_r\ge 0$ in $B_1'$,
$\zeta_r=v_r$ on $S_1$ and
\begin{equation}\label{epi}
W(\zeta_r,1)\le (1-\kappa)W(w_r,1).
\end{equation}
We note that this inequality continues to hold if as $\zeta_r$ we take
the minimizer of the functional
$W(\cdot,1)$ among all functions $\zeta\in W^{1,2}(B_1)$ with
$\zeta\big|_{S_1}=v_r\big|_{S_1}$ and $\zeta\ge 0$ in $B_1'$. Taking
such minimizer
is equivalent to saying that $\zeta_r$  is the solution of
the thin obstacle problem in $B_1$ for the Laplacian with zero thin
obstacle on $B_1'$ and boundary values $v_r$ on $S_1$. In particular, with this
choice of $\zeta_r$ we will have $W(\zeta_r,1)\le W(v_r,1)$. Next,
\begin{align*}
W(\zeta_r,1)&=\int_{B_1}|\nabla
              \zeta_r|^2-\frac{3}{2}\int_{S_1}\zeta_r^2\\
& \ge \int_{B_1}\langle A(rx)\nabla\zeta_r,\nabla\zeta_r\rangle
-\frac{3}{2}\int_{S_1}\mu(rx)\zeta_r^2 -Cr\int_{B_1}|\nabla\zeta_r|^2-Cr\int_{S_1}\zeta_r^2.
\end{align*}
If we now let $\hat{\zeta}=\zeta_r(x/r)r^{3/2}$, then on $S_r$ we have
that $\hat{\zeta}=v_r(x/r)r^{3/2}=v(x)$, and

\[
\int_{B_1}\langle A(rx)\nabla\zeta_r,\nabla\zeta_r\rangle
-\frac{3}{2}\int_{S_1}\mu(rx)\zeta_r^2=r^{-n-1}\int_{B_r}\langle A\nabla \hat{\zeta},\nabla \hat{\zeta}\rangle-\frac{3}{2}r^{-n-2}\int_{S_r}\mu \hat{\zeta}^2
\]
Since $\hat{\zeta}=v$ on $S_r$, $\hat{\zeta}\geq 0$ on $B_{r}'$ and
$v$ minimizes the energy
\eqref{energy-with-f} over $B_r$ among all such functions, we obtain
\[
\int_{B_r}\left(\langle A\nabla \hat{\zeta},\nabla \hat{\zeta}\rangle + 2f \hat{\zeta}\right)\ge
\int_{B_r}\left(\langle A\nabla v,\nabla v\rangle+ 2fv\right).
\]
Next, by \eqref{vbounds} we have $|v(x)|\leq
C|x|^{3/2}\leq C r^{3/2}$ in $B_r$. Besides, noting that $\hat{\zeta}$
solves the thin
obstacle problem in $B_r$ with boundary values $v$ on $S_r$, by
subharmonicity of $\hat{\zeta}^\pm$, we will have that 
\[
\sup_{B_r}|\hat{\zeta}|\leq \sup_{S_r}v^++\sup_{S_r}v^-\leq C r^{3/2}.
\]
Hence, we obtain
\begin{align*}
\int_{B_r}\langle A\nabla \hat{\zeta},\nabla \hat{\zeta}\rangle&\geq
\int_{B_r}\left(\langle A\nabla v,\nabla v\rangle+ 2fv -2f\hat{\zeta}\right)\\
&\geq \int_{B_r} \langle A\nabla v,\nabla v\rangle -C r^{n+(3/2)}.
\end{align*}
Combining the inequalities above, we have
\begin{equation}\label{Mbd}
\begin{aligned}
W(\zeta,1)&\ge r^{-n-1}\int_{B_r}\langle A\nabla v,\nabla v\rangle-\frac{3}{2}r^{-n-2}\int_{S_r}\mu v^2-Cr\int_{B_1}|\nabla \zeta_r|^2-Cr\int_{S_1}\zeta_r^2-Cr^{1/2}\\
&=W_L(v,r)-Cr\left(W(\zeta_r,1)+\frac{3}{2}\int_{S_1}\zeta_r^2\right)-Cr\int_{S_1}\zeta_r^2-Cr^{1/2}\\
&\ge W_L(v,r)-Cr\left(
  W(v_r,1)+\frac{3}{2}\int_{S_1}\zeta_r^2\right)-Cr\int_{S_1}v_r^2-C r^{1/2}\\
&= W_L(v,r)-Cr\left(\int_{B_1}|\nabla v_r|^2-\frac{3}{2}\int_{S_1}v_r^2+\frac{3}{2}\int_{S_1}\zeta_r^2\right)-Cr\int_{S_1}v_r^2-C r^{1/2}\\
&= W_L(v,r)-Cr\int_{B_1}|\nabla v_r|^2-Cr\int_{S_1}v_r^2-C r^{1/2}\\
&\stackrel{\eqref{vbounds}}{\ge} W_L(v,r)-Cr^{1/2}.
\end{aligned}
\end{equation}
Combining \eqref{epi} and \eqref{Mbd} we obtain
\begin{equation}
\label{compa}
\begin{aligned}
W(w_r,1)-W_L(v,r)&\ge\frac{W(\zeta_r,1)}{1-\kappa}-W_L(v,r)\ge\frac{W_L(v,r)-Cr^{1/2}}{1-\kappa}-W_L(v,r)\\&= \frac{\kappa}{1-\kappa}W_L(v,r)-Cr^{1/2}.
\end{aligned}
\end{equation}
Therefore, from \eqref{bdW'} and \eqref{compa} we conclude that
\begin{equation}\label{WWW}
\begin{aligned}
\frac{d}{dr}W_L(v,r)&\ge \frac{n+1}{r}(W(w_r,1)-W_L(v,r))-Cr^{-1/2}\\
&\ge
\frac{n+1}{r}\left(\frac{\kappa}{1-\kappa}W_L(v,r)-Cr^{1/2}\right)-C r^{-1/2}\\
&\ge\frac{n+1}{r}\frac{\kappa}{1-\kappa}W_L(v,r)-Cr^{-1/2}.
\end{aligned}
\end{equation}
Taking $\gamma\in (0,\frac12 \wedge (n+1)\frac{\kappa}{1-\kappa})$, if we use that $W_L(v,r)\geq -C r^{1/2}$, see Corollary \ref{C:weiss}, then from the above
inequality we will have
\[
\left(W_L(v,r)r^{-\gamma}\right)'\ge -Cr^{-\gamma-1/2}.
\]
Integrating in $(t,r_0)$, for $t\ge s_0$, we obtain 
\[
W_L(v,t)t^{-\gamma}\le W_L(v,r_0)r_0^{-\gamma}+Cr_0^{-\gamma+1/2}-Ct^{-\gamma+1/2}.
\]
This implies, in particular,
\begin{equation*}\label{W32}
W_L(v,t)\le Ct^{\gamma},
\end{equation*}
which establishes \eqref{WW}.

We now prove \eqref{relat}. For a fixed $x$, define
$g(r)=\frac{v(rx)}{r^{3/2}}$, so that $v_t(x)-v_s(x)=g(t)-g(s)$. Then
\begin{align}\label{bb}
\int_{S_1}|v_t-v_s| &=\int_{S_1}\left|\int_s^t
  g'(r)dr\right|
 \\
& \le \int_s^t \left(\int_{S_1}r^{-\frac{3}{2}}\left|\langle \nabla v(rx),\nu\rangle - \frac{3}{2}\frac{v(rx)}{r}\right| \right)dr
\notag\\
&=\int_s^t \left( r^{-1}\int_{S_1}| \langle \nabla v_r,\nu\rangle -\frac{3}{2}v_r| \right) dr
\notag\\
& \le C_n \int_s^t r^{-1/2} \left(r^{-1}\int_{S_1}(\langle \nabla v_r,\nu\rangle -\frac{3}{2}v_r)^2\right)^{1/2}dr.
\notag
\end{align}
By \eqref{bdW'} and \eqref{compa},
\begin{equation}\label{ma}
\begin{aligned}
\frac{1}{r}\int_{S_1}(\langle \nabla v_r,\nu\rangle -\frac{3}{2}v_r)^2& \le \frac{d}{dr}W_L(v,r)+Cr^{-1/2} -\frac{n+1}{r}(W(w_r,1)-W_L(v,r))\\
& \le \frac{d}{dr}W_L(v,r)-\frac{n+1}{r}\frac{k}{1-k}W_L(v,r)+C r^{-1/2}\\
& \le  \frac{d}{dr}W_L(v,r)+C r^{-1/2},
\end{aligned}
\end{equation}
where we have used again Corollary \ref{C:weiss}, that gives $W_L(v,r)\ge - C r^{1/2}$. This proves \eqref{relat} and completes the proof. 
\end{proof}

The next important step after Lemma~\ref{L:reg} is contained in
Proposition \ref{P:uniqreg} below. It proves that the regular
set is a relatively open set of the free boundary, and that if
$x_0\in\Gamma_{3/2}(v)$, then for $r$ small enough and $\bar{x}\in
\Gamma(v)$ in a small neighborhood of $x_0$ we can apply Lemma
\ref{L:reg} to the homogeneous scalings 
$$
v_{\bar{x},r}(x)=\frac{v_{\bar x}(rx)}{r^{3/2}}=\frac{v(\bar
  x+A^{1/2}(\bar x)rx) - b_{\bar x} rx_n}{r^{3/2}},
$$
which in turn proves the uniqueness of the blowup limits.

 \begin{prop}\label{P:uniqreg} Let $v$ solve
   \eqref{c11i}--\eqref{c44i} and $x_0\in\Gamma_{3/2}(v)$. Then, there exist constants
  $r_0=r_0(x_0)$, $\eta_0=\eta_0(x_0)>0$ such that $\Gamma(v)\cap
  B_{\eta_0}'(x_0)\subset \Gamma_{3/2}(v)$. Moreover, if $v_{\bar x, 0}$ is
  any blowup of $v$ at $\bar x\in\Gamma(u)\cap B_{\eta_0}'(x_0)$, as in Definition \ref{dfn:homog}, then 
\[
\int_{S_1}|v_{\bar x,r}-v_{\bar x, 0}|\le Cr^{\gamma},\quad \text{for
  all } r\in (0,r_0),
\]
where $C$ and  $\gamma>0$ are universal constants. In particular, the
blowup limit $v_{\bar x, 0}$ is unique.
\end{prop}

\begin{proof} Let $r_0$ and
  $\eta_0$ be as in Lemma~\ref{vxr-h}. Then, for $\bar x\in
  B_{\eta_0}'(x_0)\cap \Gamma(v)\subset \Gamma_{3/2}(v)$ and $0<r<r_0$
  consider two scalings, the homogeneous and Almgren types:
$$
v_{\bar x,r}(x)=\frac{v_{\bar x}(r x)}{r^{3/2}},\quad \tilde v_{\bar
  x,r}(x)=\frac{v_{\bar x}(r x)}{d_{\bar x, r}}
$$
By Lemma~\ref{vxr-h}, the Almgren scaling $\frac{1}{c_n}\tilde
v_{\bar x, r}\big|_{S_1}$ has the epiperimetric property in
the sense that if we extend it as a $3/2$-homogeneous
function in $B_1$, call it $w_r$, then Lemma~\ref{vxr-h} allows us to
apply the epiperimetric inequality to conclude that there exists
$\zeta_r \in W^{1,2}(B_1)$
such that $\zeta_r= \frac{1}{c_n}\tilde{v}_{\bar{x},r}$ on $S_1$,
$\zeta_r\ge 0$ in $B_1'$ and
$$
W(\zeta_r)\leq (1-\kappa) W(w_r).
$$
We next observe that if a certain function on $S_1$ has the
epiperimetric property then so does any of its multiples. In
particular, $v_{\bar x, r}\big|_{S_1}=\frac{c_nd_{\bar{x},r}}{r^{3/2}}\frac{1}{c_n}\tilde{v}_{\bar{x},r}$ has the
epiperimetric property for any $\bar x\in B_{\eta_0}'(x_0)\cap
\Gamma(v)$ and $r<r_0$. Thus, we can apply Lemma~\ref{L:reg} to
conclude that
$$
\int_{S_1} |v_{\bar x, r} -v_{\bar x, s}| \leq C r^\gamma,\quad
\text{for }0<s\leq r<r_0,
$$
for universal $C$ and $\gamma>0$. Now, if over some sequence $v_{\bar x, s_j}\to v_{\bar x, 0}$ (see Lemma~\ref{L:homblowupnonzero}), we will
obtain that
\[
\int_{S_1} |v_{\bar x, r} -v_{\bar x, 0}| \leq C r^\gamma,\quad
\text{for }0<r<r_0.\qedhere
\]
\end{proof}

We notice explicitly that up to this point we have not excluded the
possibility that the blowup $v_{\bar x,0}\equiv 0$. This is done in
the following proposition.

\begin{prop}\label{P:nondege} The unique blowup $v_{\bar x, 0}$ in
  Proposition~\ref{P:uniqreg} is nonzero.
\end{prop}

\begin{proof}
Assume to the contrary that $v_{\bar x, 0}=0$. Then, from
Proposition~\ref{P:uniqreg} we have that
\[
\int_{S_1} |v_{\bar x, r}| \le
Cr^{\gamma},\quad\text{for }0<r<r_0.
\]
But then,
\begin{equation}\label{comesfromepi}
\int_{S_1}|\tilde v_{\bar x, r}|=\int_{S_1}|v_{\bar x,
  r}|\frac{r^{3/2}}{d_{\bar x, r}}\le C\frac{r^{3/2+\gamma}}{d_{\bar x,
  r}}.
\end{equation}
On the other hand, by Lemma~\ref{L:boundonHr} for every $\ve >0$ there
exists $r_\ve>0$ such that
$$
d_{\bar x, r}=\left(\frac{1}{r^{n-1}}H_{\bar x}(r)\right)\geq c
M_{\bar x}(r)^{1/2}\geq c r^{(3+\varepsilon)/2},\quad \text{for }0<r<r_\varepsilon
$$
If we choose $\varepsilon<2\gamma$, we then obtain as $r\to 0+$ 
$$
\int_{S_1}|\tilde v_{\bar x, r}|\leq C r^{\gamma-\varepsilon/2}\to 0.
$$
Since there exists $e'\in\R^{n-1}$, and a subsequence $r=r_j\to 0+$
such that $\tilde{v}_{\bar x, r}\to c_n\Re(\langle x', e'\rangle
+i|x_n|)^{3/2}$, this is clearly a contradiction.
\end{proof} 

In what follows we denote by 
$$
v_{\bar x, 0}(x)= a_{\bar x} \Re (\langle x',
  \nu_{\bar x}\rangle + i|x_n|^{3/2}),\quad a_{\bar x}>0,\ \nu_{\bar
    x}\in S_1'
$$ the unique homogeneous blowup at $\bar x$.

\begin{prop}\label{blowups-close} Let $v$ be a solution of \eqref{c11i}--\eqref{c44i} with
  $x_0\in \Gamma_{3/2}(v)$. Then, there exists
  $\eta_0>0$ depending on $x_0$ such that
$$
\int_{S_1'}| v_{\bar x, 0}- v_{\bar y, 0}|\leq C |\bar x -\bar y|^\beta\quad
\text{for }\bar x, \bar y \in B_{\eta_0}'(x_0)\cap \Gamma(v),
$$
where $C$ and $\beta>0$ are universal constants.
\end{prop}

\begin{proof} Let $\eta_0$ and $r_0$ be as in
  Proposition~\ref{P:uniqreg}. Then,
  we will have
$$
\int_{S_1} |v_{\bar x, 0}-v_{\bar y,0}|\leq C r^\gamma + \int_{S_1}
|v_{\bar x, r}-v_{\bar y,r}|
$$
for any $r<r_0$ and $\bar x, \bar y \in B_{\eta_0}\cap
\Gamma(v)$. In this inequality we will chose $r=|\bar x-\bar y|^{\sigma}$ with $0<\sigma<1$ to be specified below.
We then estimate the integral on the right-hand side of the above
inequality. First we notice that 
\begin{align*}
  v_{\bar x}(z)&=v(\bar x+A^{1/2}(\bar x)z)-\langle A^{1/2}(\bar x)\nabla
  v(\bar x),e_n \rangle z_n\\
&= v(\bar x+A^{1/2}(\bar x)z)-\D_n v(\bar x)\langle{\bar x+A^{1/2}(\bar x) z, e_n}\rangle
\end{align*}
by using property \eqref{mjai} of the coefficient matrix
$A$. Therefore, denoting 
\begin{align*}
\xi(s)&=[s\bar x+(1-s)\bar
y]+[sA^{1/2}(\bar x)+(1-s)A^{1/2}(\bar y)]z\\
p(s)&=s\D_nv(\bar x)+(1-s)\D_n v(\bar
  y)
\end{align*}
we will have
\begin{align*}
v_{\bar x}(z)-v_{\bar y}(z)&=\int_0^1\frac{d}{ds}\left(v(\xi(s))-p(s)\langle\xi(s),e_n\rangle\right)ds\\
& =\int_0^1 \bigl(\langle \nabla v(\xi(s))- p(s)e_n, [\bar x-\bar
  y]+[A^{1/2}(\bar x)-A^{1/2}(\bar y)]z\rangle\\
&\qquad\mbox{}- [\D_n v(\bar x)-\D_n
  v(\bar y)]\langle [sA^{1/2}(\bar x)+(1-s)A^{1/2}(\bar y)]z ,e_n\rangle\bigr)ds
\end{align*}
where for the last term we have used the orthogonality  $\langle{[s\bar x+(1-s)\bar
y]}, e_n\rangle=0$. This gives
\begin{align*}
|v_{\bar x}(z)-v_{\bar y}(z)|&\leq C(|\bar x-\bar y|+|z|)^{1/2}(|\bar
x-\bar y| + |z||\bar x-\bar y|)+C|\bar
x-\bar y|^{1/2}|z|.
\end{align*}
Using the above estimate we then obtain
\begin{align*}
\int_{S_1}|v_{\bar x, r}-v_{\bar y,r}|&=\int_{S_1} \frac{|v_{\bar
    x}(rz)-v_{\bar y}(rz)|}{r^{3/2}}\\
&\leq C\frac{(|\bar x-\bar y|+ r)^{1/2}|\bar x-\bar y|(1+r)+|\bar
  x-\bar y|^{1/2}r}{r^{3/2}}\\
&\leq C|\bar x-\bar y|^{1-\sigma}+C|\bar x-\bar y|^{(1-\sigma)/2}\leq C|\bar x-\bar y|^{(1-\sigma)/2},
\end{align*}
if we choose $r=|\bar x-\bar y|^{\sigma}$ with $0<\sigma<1$.

Going back to the beginning of the proof, we conclude
$$
\int_{S_1} |v_{\bar x, 0}-v_{\bar y,0}|\leq C (|\bar x-\bar
y|^{\sigma\gamma}+ |\bar x-\bar y|^{(1-\sigma)/2})=C|\bar x-\bar
y|^{2\beta},
$$
with $2\beta={\gamma/(1+2\gamma)}$ if we choose $\sigma=1/(1+2\gamma)$.

It remains to show that we can change the integration over
$(n-1)$-dimensional sphere $S_1$ to $(n-2)$-dimensional $S_1'$. To
this end, we note that both $v_{\bar x, 0}$ and $v_{\bar y,0}$ are solutions
of the Signorini problem for the Laplacian with zero thin obstacle and
therefore
$$
\Delta(v_{\bar x, 0}-v_{\bar y,0})_\pm\geq 0\footnotemark{}.
$$
\footnotetext{\emph{Short proof:} the only nontrivial place to verify
  the subharmonicity
of $(v_{\bar x, 0}-v_{\bar y,0})_+$ is at points $z\in \R^{n-1}$
with $v_{\bar x, 0}(z)>0$ and $v_{\bar y, 0}(z)=0$; but near
such $z$, $\Delta v_{\bar x, 0}=0$ and $\Delta v_{\bar y, 0}\leq 0$ by
the Signorini conditions on $\R^{n-1}$.}%
Thus by the energy inequality we obtain that
$$
\int_{B_1}|\nabla (v_{\bar x, 0}-\nabla
v_{\bar y,0})|^2\leq C\int_{B_1}|(v_{\bar x, 0}-v_{\bar y,0})|^2\leq C
|\bar x-\bar y|^{2\beta}.
$$
(Recall that $v_{\bar z,0}$ are $3/2$-homogeneous functions with
uniform $C^{1,1/2}$ estimates). Then, using the trace inequality, we obtain that
$$
\int_{B_{1/2}'} |v_{\bar x, 0}-v_{\bar y,0}|^2\leq C |\bar x-\bar y|^{2\beta}
$$
This is equivalent to 
\[
\int_{S_1'} |v_{\bar x, 0}- v_{\bar y,0}|^2\leq C |\bar x-\bar y|^{2\beta},
\]
and using H\"older's inequality
\[
\int_{S_1'} |v_{\bar x, 0}- v_{\bar y,0}|\leq C |\bar x-\bar y|^{\beta},
\]
as claimed.
\end{proof}

\begin{lemma}\label{a-nu-Holder} Let $v$ be as in Proposition~\ref{blowups-close} and
  $v_{\bar x, 0}= a_{\bar x} \Re (\langle x',\nu_{\bar x}\rangle + i|x_n|)^{3/2}$ be the unique homogeneous blowup at
  $\bar x\in B_{\eta_0}'(x_0)\cap \Gamma(v)$. Then, for a constant
  $C_0$ depending on $x_0$ we have
\begin{align}\label{aHolder}
|a_{\bar x}-a_{\bar y}|&\leq C_0|\bar x-\bar y|^{\beta}\\
\label{nuHolder}
|\nu_{\bar x}-\nu_{\bar y}|&\leq C_0|\bar x-\bar y|^{\beta},
\end{align}
for $\bar x,\bar y\in B_{\eta_0}'(x_0)\cap \Gamma(v)$.
\end{lemma}
\begin{proof} The first inequality follows from the observation that
$$
C_n a_{\bar x}=\|v_{\bar x, 0}\|_{L^1(S_1')},\quad C_n a_{\bar
  y}=\|v_{\bar y, 0}\|_{L^1(S_1')}
$$
with the same dimensional constant $C_n$ and therefore
$$
C_n|a_{\bar x}-a_{\bar y}|\leq \|v_{\bar x, 0}-v_{\bar
  y,0}\|_{L^1(S_1')}\leq C|\bar x-\bar y|^{\beta},
$$
which establishes \eqref{aHolder}.
To prove \eqref{nuHolder}, we first note that by \eqref{aHolder}
$$
\|a_{\bar x}\langle z,\nu_{\bar x}\rangle_+^{3/2}-a_{\bar y}\langle
z,\nu_{\bar y}\rangle_+^{3/2}\|_{L^1(S_1')}\leq C |\bar x-\bar y|^\beta
$$
and therefore we obtain
$$
\|\langle z,\nu_{\bar x}\rangle_+^{3/2}-\langle
z,\nu_{\bar y}\rangle_+^{3/2}\|_{L^1(S_1')}\leq C_0 |\bar x-\bar y|^\beta.
$$
(Here we have used that by \eqref{aHolder} we may assume that $a_{\bar
  y}> a_{x_0}/2$ if $\eta_0$ is small).

On the other hand, it is easy to see from geometric
considerations\footnotemark{} that 
$$
\|\langle z,\nu_{\bar x}\rangle_+^{3/2}-\langle
z,\nu_{\bar y}\rangle_+^{3/2}\|_{L^1(S_1')}\geq c_n|\nu_{\bar
  x}-\nu_{\bar y}|
$$
\footnotetext{just notice that 
$
\D_\theta\int\limits_{\substack{\langle
    z,e_1\rangle\geq 1/2\\ \langle z,e_2\rangle\geq 1/2}}
\langle z,\cos\theta e_1+\sin\theta
e_2\rangle^{3/2}\big|_{\theta=0}=\frac32\int\limits_{\substack{\langle
    z,e_1\rangle\geq 1/2\\ \langle z,e_2\rangle\geq 1/2}} \langle z,e_1\rangle^{1/2}\langle
z,e_2\rangle >0
$}
implying that
\[
c_n|\nu_{\bar x}-\nu_{\bar y}|\leq C_0 |\bar x-\bar y|^\beta.\qedhere
\]
\end{proof}

\begin{theo}\label{free-bound-reg} Let $v$ be a solution of
  \eqref{c11i}--\eqref{c44i} with $x_0\in \Gamma_{3/2}(v)$. Then there
  exists a positive $\eta_0$ depending on $x_0$ such that
  $B'_{\eta_0}(x_0)\cap \Gamma(v)\subset \Gamma_{3/2}(v)$ and
$$
B'_{\eta_0}\cap \Lambda(v)=B'_{\eta_0}\cap \{x_{n-1}\leq g(x'')\}
$$
for $g\in C^{1,\beta}(\R^{n-2})$ with a universal exponent
$\beta\in(0,1)$, after a possible rotation of coordinate axes in $\R^{n-1}$.
\end{theo}
\begin{proof} We subdivide the proof into several steps.

\medskip
\noindent \emph{Step 1.} Let $\eta_0$ be as Proposition~\ref{P:uniqreg}. We then claim that for any $\varepsilon>0$ there is $r_\varepsilon>0$ such that
$$
\|v_{\bar x, r}-v_{\bar x,0}\|_{C^1(B_1^\pm\cup B_1')}<\varepsilon,\quad\text{for }\bar
x\in B'_{\eta_0/2}(x_0)\cap \Gamma(v),\ r<r_\varepsilon.
$$
Indeed, arguing by contradiction, we will have a sequence of points
$\bar x_j\in  B'_{\eta_0/2}(x_0)\cap \Gamma(v)$ and radii $r_j\to 0$
such that
$$
\|v_{\bar x_j, r_j}-v_{\bar x_j,0}\|_{C^1(B_1^\pm\cup B_1')}\geq \varepsilon_0
$$
for some $\varepsilon_0>0$. Clearly, we may assume that $\bar
x_j\to \bar x_0\in \overline {B_{\eta_0/2}'(x_0)}\cap \Gamma(v)$. Now,
the scalings $v_{\bar x_j, r_j}$ are
uniformly bounded in $C^{1,1/2}(B_R^\pm\cup B'_R)$ for any $R>0$ and
thus we may assume that 
$$
v_{\bar x_j, r_j}\to w\quad\text{in }C^1_\loc((\R^n)^\pm\cup
\R^{n-1}\}.
$$
We claim that actually $w=v_{\bar x_0,0}$.  Indeed, by integrating the
inequality in Proposition~\ref{P:uniqreg}, we will
  have 
$$
\|v_{\bar x, r}-v_{\bar x, 0}\|_{L^1(B_R)}\leq C_R r^\gamma, \quad \text{for }\bar
x\in B'_{\eta_0}(x_0)\cap \Gamma(v),\ r<r_0/R,
$$
which will immediately imply that $v_{\bar x_j, 0}\to w$ in
$L^1(B_R)$. On the other hand, Lemma~\ref{a-nu-Holder} implies that
$v_{\bar x_j, 0}\to v_{\bar x_0, 0}$ in $C^1(B_R^\pm\cup B_R')$. Hence
$w=v_{\bar x_0, 0}$. Moreover, we get that both $v_{\bar x_j,r_j}$ and
$v_{\bar x_j,0}$ converge in $C^1(B_1^\pm\cup B_1')$ to the same function
$v_{\bar x_0,0}$ and therefore
$$
\|v_{\bar x_j, r_j}-v_{\bar x_j,0}\|_{C^1 (B_1^\pm\cup B_1')}\to 0
$$
contrary to our assumption.

\medskip
\noindent \emph{Step 2}. For a given $\varepsilon>0$ and a unit vector
$\nu\in\R^{n-1}$ define the cone 
$$
\mathcal{C}_\varepsilon(\nu)=\{x'\in\R^{n-1}\mid \langle x',\nu\rangle \geq \varepsilon |x'|\}
$$
We then claim that for any $\varepsilon>0$
  there exists $r_\varepsilon$ such that for any $\bar x\in
  B'_{\eta_0/2}(x_0)\cap \Gamma(v)$ we have
$$
\mathcal{C}_\varepsilon(\nu_{\bar x})\cap B'_{r_\varepsilon}\subset
\{v_{\bar x}(\cdot, 0)>0\}.
$$
Indeed, consider a cutout from the sphere $S'_{1/2}$ by the cone $\mathcal{C}_\varepsilon(\nu)$ 
$$
K_\varepsilon(\nu)= \mathcal{C}_\varepsilon(\nu)\cap S'_{1/2}.
$$
Note that 
$$
K_\varepsilon(\nu_{\bar x})\Subset\{v_{\bar x, 0}(\cdot, 0)>0\}\cap B_1'\quad\text{and}\quad
v_{\bar x, 0}(\cdot, 0)\geq a_{\bar x} c_\varepsilon\quad\text{on
}K_\varepsilon(\nu_{\bar x})
$$
for some universal $c_\varepsilon>0$. Without loss of generality, by Lemma~\ref{a-nu-Holder}, we may assume
that $a_{\bar x}\geq a_0/2$ for $\bar x\in B'_{\eta_0}(x_0)\cap
\Gamma(v)$. Then, applying Step 1 above, we can find $r_\varepsilon>0$
such that
$$
v_{\bar x, r}(\cdot, 0)> 0\quad\text{on } K_\varepsilon (\nu_{\bar
  x}),\quad\text{for all }r<r_\varepsilon
$$
Scaling back by $r$, we have
$$
v_{\bar x}(\cdot, 0)> 0\quad\text{on}\quad r K_{\varepsilon}(\nu_{\bar x})=\mathcal{C}_\varepsilon(\nu)\cap S'_{r/2},\quad r<r_\varepsilon
$$
Taking the union over all $r<r_\varepsilon$, we obtain
$$
\mathcal{C}_\varepsilon(\nu)\cap
B_{r_\varepsilon/2}'\subset \{v_{\bar x}(\cdot, 0) >0\}.
$$

\medskip\noindent
\emph{Step 3.} We next claim that for any $\varepsilon>0$
  there exists $r_\varepsilon$ such that for any $\bar x\in
  B'_{\eta_0/2}(x_0)\cap \Gamma(v)$ we have
$$
- \mathcal{C}_\varepsilon(\nu_{\bar x})\cap B'_{r_\varepsilon}\subset
\{v_{\bar x}(\cdot,0)=0\}.
$$
To prove that we note that
$$ 
- K_\varepsilon(\nu_{\bar x})\Subset\{v_{\bar x, 0}(\cdot, 0)=0\}\cap B_1'\quad\text{and}\quad
(\partial_{x_n}^--\partial_{x_n}^+)v_{\bar x, 0}(\cdot, 0)\geq a_{\bar x} c_\varepsilon>(a_{0}/2)c_\varepsilon\quad\text{on
}- K_\varepsilon(\nu_{\bar x})
$$
for a universal $c_\varepsilon>0$. Then, from Step 1, we will have the
existence of $r_\varepsilon>0$ such that
$$
\langle A_{\bar x}(rx)\nabla v_{\bar x, r}(x),e_{n}^-\rangle +\langle A_{\bar x}(rx)\nabla v_{\bar x, r}(x),e_{n}^+\rangle >0\quad\text{on }-K_{\varepsilon}(\nu_{\bar x}),\quad\text{for all }r<r_\varepsilon.
$$
Thus, from the complementarity
conditions, we will have
$$
v_{\bar x, r}(\cdot,0) =0\quad\text{on }-K_{\varepsilon}(\nu_{\bar x}),\quad\text{for all }r<r_\varepsilon.
$$
Arguing as in the end of Step 2, we conclude that
$$
-\mathcal{C}_\varepsilon(\nu)\cap
B_{r_\varepsilon/2}'\subset \{v_{\bar x}(\cdot, 0) =0\}.
$$

\medskip\noindent
\emph{Step 4.} Here without loss of generality we will assume that
$A(x_0)=I$ and $\nu_{x_0}=e_{n-1}$. 
Changing from function $v_{\bar x}$ to $v$, we may rewrite the result
of Steps 2 and 3 as 
\begin{align*}
\bar x+ A(\bar x)^{1/2} (\mathcal{C}_\varepsilon (\nu_{\bar x})\cap
B_{r_\varepsilon/2}')& \subset
\{v(\cdot, 0)>0\},\\
\bar x- A(\bar x)^{1/2} (\mathcal{C}_\varepsilon (\nu_{\bar x})\cap
B_{r_\varepsilon/2}')& \subset
\{v(\cdot, 0)=0\},
\end{align*}
for any $\bar x\in B'_{\eta_0}(x_0)\cap \Gamma(v)$.
Taking $\bar x$ sufficiently close to $x_0$ (and using
Lemma~\ref{a-nu-Holder}) we can guarantee that 
$$
 A(\bar x)^{1/2} (\mathcal{C}_\varepsilon (\nu_{\bar x})\cap
 B_{r_\varepsilon/2}')\supset \mathcal{C}_{2\varepsilon}(e_{n-1})\cap B_{r_\varepsilon/4}'.
$$
Hence, there exists $\eta_\varepsilon>0$ such that
\begin{equation}\label{cones-inclusion}
\begin{aligned}
\bar x+ (\mathcal{C}_{2\varepsilon} (e_{n-1})\cap B_{r_\varepsilon/4}')&\subset
\{v(\cdot,0)>0\},\\
\bar x- (\mathcal{C}_{2\varepsilon} (e_{n-1})\cap B_{r_\varepsilon/4}')&\subset
\{v(\cdot,0)=0\}
\end{aligned}
\end{equation}
for any $\bar x\in B'_{\eta_\varepsilon}(x_0)\cap
\Gamma(v)$.
Now, fixing $\varepsilon=\varepsilon_0$, by the standard arguments, we can conclude that there exists a
Lipschitz function $g:\R^{n-2}\to \R$ with $|\nabla g|\leq C_n/\varepsilon_0$
such that
\begin{align*}
B'_{\eta_{\varepsilon_0}}(x_0)\cap\{v(\cdot,
0)=0\}&=B'_{\eta_{\varepsilon_0}}(x_0)\cap \{x_{n-1}\leq g(x'')\}\\
B'_{\eta_{\varepsilon_0}}(x_0)\cap\{v(\cdot,
0)>0\}&=B'_{\eta_{\varepsilon_0}}(x_0)\cap \{x_{n-1}>g(x'')\}
\end{align*}

\medskip\noindent
\emph{Step 5.} Using the normalization $A(x_0)=I$ and
$\nu_{x_0}=e_{n-1}$ as in Step 4 and letting $\varepsilon\to 0$ we see
that $\Gamma(v)$ is differentiable at $x_0$ with normal
$\nu_0$. Recentering at any $\bar x\in B'_{\eta_{\varepsilon_0}}(x_0)\cap
\Gamma(v)$, we see that $\Gamma(v)$ has a normal 
$$
A(\bar x)^{-1/2}\nu_{\bar x}
$$
at $\bar x$. Finally noting that by Lemma~\ref{a-nu-Holder} the
mapping $\bar x\mapsto A(\bar x)^{-1/2}\nu_{\bar x}$ is $C^\beta$, we obtain that the function $g$ in Step 4
is $C^{1,\beta}$. 

The proof is complete.
\end{proof}

\end{document}